\documentclass[reqno]{llncs}
\usepackage[utf8]{inputenc}
\usepackage{amsmath}
\usepackage{amsfonts}
\usepackage{amssymb}
\usepackage{enumerate}
\usepackage{subcaption}
\usepackage{algorithm}
\usepackage{tikz}
\usetikzlibrary{shapes}
\usetikzlibrary{arrows}
\usepackage{here}
\usepackage[shortlabels]{enumitem}

\usepackage[left=1.25in, right=1.25in, top=.9in, bottom=.9in]{geometry}
\usepackage[colorlinks=true,breaklinks=true,bookmarks=true,urlcolor=blue,
     citecolor=blue,linkcolor=blue,bookmarksopen=false,draft=false]{hyperref}
\usepackage{cleveref}

\newcommand{\Z}{{\mathbb Z}}

\newcommand{\R}{\mathbb R}

\newcommand\veb{{\ve b}}

\newcommand\ved{{\ve d}}

\newcommand\veg{{\ve g}}

\newcommand\vex{{\ve x}}

\newcommand{\CDG}{{CDG(\mathcal{C}, \mathcal{C}')}}
\newcommand{\C}{{\mathcal{C}}}
\newcommand{\PP}{{BPP(\kappa^+, \kappa^-)}}

\newcommand{\eoproof}{\hspace*{\fill} $\square$ \vspace{5pt}}

\def\ve#1{\mathchoice{\mbox{\boldmath$\displaystyle\bf#1$}}
{\mbox{\boldmath$\textstyle\bf#1$}}
{\mbox{\boldmath$\scriptstyle\bf#1$}}
{\mbox{\boldmath$\scriptscriptstyle\bf#1$}}}

\title{Constructing Clustering Transformations}

\author{Steffen Borgwardt\inst{1}\and Charles Viss\inst{2}}

\institute{\email{\href{mailto:steffen.borgwardt@ucdenver.edu}{steffen.borgwardt@ucdenver.edu}};
University of Colorado Denver \and
\email{\href{mailto:charles.viss@ucdenver.edu}{charles.viss@ucdenver.edu}};
University of Colorado Denver 
}

\date{\today}

\begin{document}

\maketitle

\begin{abstract}
Clustering is one of the fundamental tasks in data analytics and machine learning. In many situations, different clusterings of the same data set become relevant. For example, different algorithms for the same clustering task may return dramatically different solutions. We are interested in applications in which one clustering has to be transformed into another; e.g., when a gradual transition from an old solution to a new one is required.

In this paper, we devise methods for constructing such a transition based on linear programming and network theory. We use a so-called clustering-difference graph to model the desired transformation and provide methods for decomposing the graph into a sequence of elementary moves that accomplishes the transformation. These moves are equivalent to the edge directions, or circuits, of the underlying partition polytopes. Therefore, in addition to a conceptually new metric for measuring the distance between clusterings, we provide new bounds on the circuit diameter of these partition polytopes.
\end{abstract}


\section{Introduction and Preliminaries}\label{sec:introduction}
Clusterings of large data sets play an important role in data analytics, machine learning, and informed decision-making in general. In many applications, there exists a desired clustering corresponding to an optimal solution to an optimization problem. However, directly implementing such a solution can be challenging -- instead, a gradual sequence of transitions which transforms an initial, sub-optimal solution into the improved clustering is desired.

Consider the example in land consolidation from \cite{bbg-13,bbg-14,bbg-15}. In a Bavarian agricultural region, 471 lots are cultivated by 7 different farmers. The initial distribution of the ownership of the lots, depicted in Figure~\ref{fig:lots1}, is quite problematic -- the scattered, small lots result in large transportation overhead and prohibit the use of heavy machinery. To address this, the authors worked with the Bavarian State to facilitate a voluntary land exchange among the farmers in which the boundaries of the lots would remain the same while the cultivation rights for the lots would be redistributed.

The combinatorial redistribution of lots can be modeled as a clustering problem: a set of data points (the lots) must be partitioned into clusters (the farmers) under the restriction that each farmer keeps his original total value of land. This restriction makes it provably hard to determine an optimal redistribution of the lots, but it is possible to compute an approximation of the global optimum based on linear programming over projections of transportation polytopes \cite{bbg-14}. A computed solution in which the lots form large, connected sections of land is depicted in Figure~\ref{fig:lots2}. 

\begin{figure}[h]
\center
\begin{subfigure}[t]{0.45 \textwidth}
\center
\includegraphics[width=0.99\textwidth]{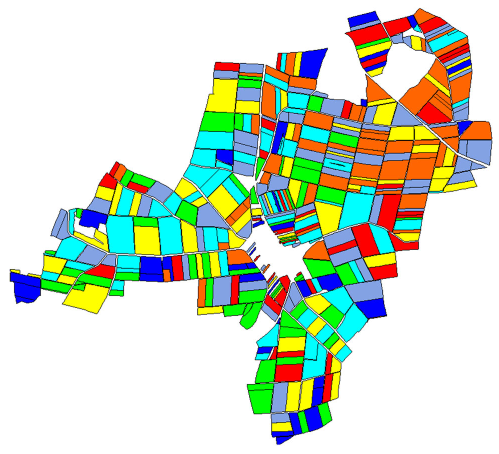}
\caption{Original clustering of lots.}\label{fig:lots1}
\end{subfigure}
\qquad
\begin{subfigure}[t]{0.45 \textwidth}
\center
\includegraphics[width=0.99\textwidth]{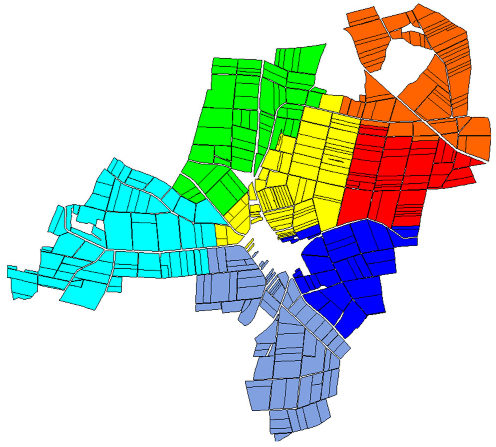}
\caption{Desired clustering of lots.}\label{fig:lots2}
\end{subfigure}
\caption{Two clusterings of 471 agricultural lots among 7 farmers (images from \cite{bbg-13,bbg-14,bbg-15}).}
\end{figure}

However, such a radical redistribution of lots among the farmers (more than 70\% of the lots change ownership from Figure~\ref{fig:lots1} to \ref{fig:lots2}) cannot realistically take place all at once. Crop rotations, required machinery, and other processes for farming stability must be respected. Therefore, the farmers requested a ``best" way to gradually implement the proposed changes over the course of several years. In this paper, we propose methods for constructing such a transformation between clusterings based on linear programming and network theory.

The need for the construction of an efficient gradual transition to a new, given clustering also arises in many other applications. For example, an insurance company may want to gradually transition their customers to a new clustering of premium classes. In other situations, there are snapshots of the same data set at different times and the goal is to devise a model that explains how the data gradually changed over time.

In general, we consider partitions of a data set $X := \{x_1,...,x_n\}$ into $k$ labeled clusters where each item $x_i$ is assigned to exactly one of the clusters. We call such a partition $\mathcal{C} := (C_1,...,C_k)$ a \textit{$k$-clustering} of $X$, or simply a \textit{clustering} of $X$ when $k$ is clear from context. As is the case in most clustering applications, we assume that the number of items is significantly greater than the number of clusters; i.e., $n \gg k$.

We additionally consider situations in which upper and lower bounds are given for the sizes of the clusters. Such bounds may arise directly from an application itself or may be introduced to guide clustering algorithms to return sufficiently balanced solutions. Specifically, let $\kappa^+, \kappa^- \in \Z_+^k$ with $\kappa^+ \geq \kappa^-$ be given. A \textit{bounded-size $k$-clustering} of $X$ with respect to $\kappa^+$ and $\kappa^-$ then satisfies $\kappa_i^- \leq |C_i| \leq \kappa_i^+$ for $i = 1,...,k$. This concept generalizes the \textit{fixed-size $k$-clusterings} from \cite{b-13} in which each cluster contains a fixed number of items (i.e., a bounded-size $k$-clustering with $\kappa^+ = \kappa^-$). The classical transportation problem and the assignment problem, along with their related polytopes, are well-studied topics in optimization corresponding to fixed-size clusterings \cite{br-74,dkos-09,kw-68}.

For a data set $X$ and cluster size bounds $\kappa^+, \kappa^- \in \Z_+^k$, the \textit{bounded-size partition polytope} $\PP$ models the set of all bounded-size $k$-clusterings of $X$ with respect to $\kappa^+$ and $\kappa^-$ \cite{bv-17}. Specifically, for $i=1,...,k$ and $j=1,...,n$, let $y_{ij}$ indicate whether or not cluster $C_i$ receives item $x_j$ in a $k$-clustering $\mathcal{C} = (C_1,...,C_k)$ of $X$. Then $BPP(\kappa^+, \kappa^-)$ is given by the following system of constraints:
\begin{align*}
\sum_{i=1}^k y_{ij} &= 1 \ \ \ \ \  \  j=1,...,n \\
\sum_{j=1}^n y_{ij} &\geq \kappa_i^-  \ \ \ \ i=1,...,k\\
\sum_{j=1}^n y_{ij} &\leq \kappa_i^+  \ \ \ \  i=1,...,k \\
y_{ij} & \geq 0 \ \ \ \ \  \ i=1,...,k, \  j=1,...,n.
\end{align*}
Since these constraints form a totally unimodular matrix, the right-hand side is integral, and each $y_{ij}$ is implicitly bounded between 0 and 1, $BPP(\kappa^+, \kappa^-)$ is in fact a 0/1-polytope whose vertices correspond to the feasible $k$-clusterings of $X$. For $\kappa^+ = \kappa^-$, this polytope generalizes the \textit{fixed-size partition polytope} from \cite{b-13}. It is also an instance of the bounded-shape partition polytope from \cite{bh-17} when $X$ is the standard basis of $\R^n$. 

In the fixed-size partition polytope, the edges also have a combinatorial interpretation: two vertices share an edge if and only if the corresponding clusterings differ by a single cyclical move of items among the clusters, formally defined in Section~\ref{sec:moves}. This fact can be used to prove new bounds on the \textit{combinatorial diameter} of the polytope -- the maximum length of a shortest edge walk between any pair of vertices -- and provide practical methods for constructing transformations between fixed-size clusterings \cite{b-13}. In this paper, we generalize these methods to the bounded-size partition polytope $BPP$. Although the edges of this polytope have a more technical characterization, its \textit{circuits} correspond to a set of natural cyclical and sequential moves of items among the clusters \cite{bv-17}. However, constructing transformations using these two types of moves is significantly more challenging than using only cyclical moves due to their different effects on the sizes of the underlying clusters. 

Circuits, introduced as the elementary vectors of a subspace by Rockafellar \cite{r-69}, play a fundamental role in the theory of linear programming. For a general polyhedron $P = \{ \vex \in \R^n \colon A \vex = \veb, B \vex \leq \ved \}$, the set of circuits of $P$ consists of all $\veg \in \ker(A) \setminus \{ \ve 0 \}$ normalized to coprime integer components for which $B \veg$ is support-minimal over $\{B \vex \colon \vex \in \ker(A) \setminus \{\ve0 \} \}$. Geometrically, circuits correspond to all potential edge directions of $P$ as the right-hand side vectors $\veb$ and $\ved$ vary. This implies that the set of circuits serves as a universal test set for any linear program over $P$ \cite{g-75}. Hence, circuits are used as the step directions in several augmentation algorithms for solving linear programs \cite{bv-18,dhl-15,hor-13,how-11}. 

Additionally, for polyhedra such as $BPP$ defined by totally unimodular matrices, circuits have combinatorial interpretations in terms of the underlying problem. The support-minimality of circuits implies that, in a sense, steps taken in circuit directions are as simple as possible while maintaining feasibility. In combination with highly-structured problems from combinatorial optimization, these steps become particularly intuitive and are guaranteed to only visit integral points \cite{bv-17}.

Therefore, in this paper, we choose the circuits of $BPP$ as the elementary moves for transforming $k$-clusterings. In building such transformations using these circuits, we construct \textit{circuit walks} between the corresponding vertices in the polytope. As a generalization of edge walks, circuit walks are of interest due to their relationship to the combinatorial diameter of polyhedra \cite{bfh-14,bdf-16,bsy-18}. The \textit{circuit distance} between vertices refers to the minimum number of steps needed to join the vertices by a circuit walk and hence provides a lower bound on the combinatorial distance between the vertices. The related  \textit{circuit diameter} of a polyhedron -- the minimum number of steps needed to join any pair of vertices by a circuit walk -- then serves as a lower bound on its combinatorial diameter. In particular, circuit diameters provide insight into the \textit{polynomial Hirsch Conjecture}, one of the fundamental open questions in linear programming. See \cite{ks-10} for a survey of this field of study.

To build transformations between clusterings, we use one of the main tools from \cite{b-13} for analyzing fixed-size clusterings. Given two $k$-clusterings $\mathcal{C}, \mathcal{C}'$ of the same data set, the \textit{clustering-difference graph} $CDG(\mathcal{C}, \mathcal{C}')$ is a directed graph that models the transfers of items required for transforming $\mathcal{C}$ into $\mathcal{C}'$. In the context of  fixed-size clusterings, such a graph decomposes into directed cycles. Sets of cyclical moves can then be integrated together in order to bound the combinatorial distance between vertices in the fixed-size partition polytope \cite{b-13}.

We generalize this graph-theoretic approach for constructing transformations between clusterings to the context of general and bounded-size $k$-clusterings. In these situations, both cycles and paths in a clustering-difference graph correspond to circuits of the related polytopes. Integrating these different types of moves becomes much more technically challenging than integrating only cyclical moves since sequential moves alter the sizes of the underlying clusters. In Section~\ref{sec:moves}, we show how this is possible for certain combinations of cyclical and sequential moves, providing various \textit{double-moves} which reduce the number of steps needed for transforming clusterings (Theorems~\ref{thm:two_exchanges} and \ref{thm:three_intersections}). Next, we prove in Section~\ref{sec:bounds} how these double-moves lead to an upper bound (Theorem~\ref{thm:improved_bound}) on the so-called \textit{transformation distance} between $k$-clusterings: a relaxation of the circuit distance. In Section~\ref{sec:diameter} we then prove the implications of this bound on the circuit diameter of the bounded-size partition polytope (Theorem~\ref{thm:improved_diameter}). We end with a brief discussion on future directions of research in Section~\ref{sec:future_directions}.

\section{Moves and Double-Moves for Transforming Clusterings}\label{sec:moves}

Let $\mathcal{C}, \mathcal{C}'$ be two $k$-clusterings of the same data set. We recall from \cite{b-13} the definition of the \textit{clustering-difference graph} $\CDG$ from $\mathcal{C}$ to $\mathcal{C'}$, a graph-theoretic model for the difference between the two clusterings.

\begin{definition}[Clustering-difference Graph]\label{def:cdg}
For two $k$-clusterings $\mathcal{C} = (C_1,...,C_k)$ and $\C' = (C'_1,...,C'_k)$ of a data set $X = \{x_1,...,x_n\}$, the clustering-difference graph $\CDG$ from $\mathcal{C}$ to $\mathcal{C}'$ is a directed arc-labeled multigraph with vertex set $V = \{c_1,...,c_k\}$ and edge set $A$, where an edge $(c_i, c_j)$ with label $x_\ell$ belongs to $A$ if and only if $x_\ell \in C_i$ and $x_\ell \in C'_j$ for $i \neq j$. 
\end{definition}

\noindent Thus, the edges of $\CDG$ describe all single-item \textit{transfers} needed to transform $\C$ into $\C'$. The number of edges is equal to the number of items whose cluster assignment differs from $\C$ to $\C'$ -- if an edge $(c_i, c_j)$ with label $x_\ell$ belongs to $A$, item $x_\ell$ must be moved from cluster $i$ to cluster $j$ as a part of the transformation. Note that this allows for parallel edges in $\CDG$, but all edges have different labels. We refer to the transfers of a $CDG$ consisting of a single directed cycle as a \textit{cyclical move} of items among the clusters. Similarly, the set of transfers from a directed path is referred to as a \textit{sequential move} of items. See Figure~\ref{fig:cdgs} for examples of these two types of moves with corresponding clustering-difference graphs.

\begin{figure}
\center
\begin{subfigure}[t]{0.45 \textwidth}
\center
\begin{tikzpicture}[vertices/.style={draw, fill=black, circle, inner sep=0pt, minimum size = 4pt, outer sep=0pt}, scale=1.3]

\node[vertices] (c_1) at (1.5,1.1) {};
\node[vertices] (c_2) at (2.5, 2) {};
\node[vertices] (c_3) at (3.5, 1.1) {};
\node[vertices] (c_4) at (3.1, 0) {};
\node[vertices] (c_5) at (1.9, 0) {};

\foreach \to/\from in {c_1/c_2, c_2/c_3, c_3/c_4, c_4/c_5, c_5/c_1}
\draw[draw=black, line width= 1,  ->, >=latex]  (\to)--(\from);

\node[above left] at (c_1) {$c_1$};
\node[above] at (c_2) {$c_2$};
\node[above right] at (c_3) {$c_3$};
\node[below right] at (c_4) {$c_4$};
\node[below left] at (c_5) {$c_5$};

\node (x_1) at (1.95, 1.69) {$x_1$};
\node (x_2) at (3.05, 1.69) {$x_2$};
\node (x_3) at (3.45, 0.5) {$x_3$};
\node (x_4) at (2.5, -0.13) {$x_4$};
\node (x_5) at (1.55, 0.5) {$x_5$};

\end{tikzpicture}
\caption{The clustering-difference graph $\CDG$ when $\C' = ( \{x_5\}, \{x_1\}, \{x_2\}, \{x_3\}, \{x_4\})$. The graph is a directed cycle, representing a cyclical move of items among the clusters.}\label{fig:cyclical}
\end{subfigure}
\qquad
\begin{subfigure}[t]{0.45 \textwidth}
\center
\begin{tikzpicture}[vertices/.style={draw, fill=black, circle, inner sep=0pt, minimum size = 4pt, outer sep=0pt}, scale=1.3]

\node[vertices] (c_1) at (1.5,1.1) {};
\node[vertices] (c_2) at (2.5, 2) {};
\node[vertices] (c_3) at (3.5, 1.1) {};
\node[vertices] (c_4) at (3.1, 0) {};
\node[vertices] (c_5) at (1.9, 0) {};

\foreach \to/\from in {c_1/c_2, c_2/c_3, c_3/c_4, c_4/c_5}
\draw[draw=black, line width= 1,  ->, >=latex]  (\to)--(\from);

\node[above left] at (c_1) {$c_1$};
\node[above] at (c_2) {$c_2$};
\node[above right] at (c_3) {$c_3$};
\node[below right] at (c_4) {$c_4$};
\node[below left] at (c_5) {$c_5$};

\node (x_1) at (1.95, 1.69) {$x_1$};
\node (x_2) at (3.05, 1.69) {$x_2$};
\node (x_3) at (3.45, 0.5) {$x_3$};
\node (x_4) at (2.5, -0.13) {$x_4$};

\end{tikzpicture}
\caption{The clustering-difference graph $\CDG$ when $\C' = ( \emptyset, \{x_1\}, \{x_2\}, \{x_3\}, \{x_4, x_5\})$. The graph is a directed path, representing a sequential move of items among the clusters.}\label{fig:sequential}
\end{subfigure}
\caption{Given $X = \{x_1,x_2,x_3,x_4,x_5\}$, let $\C = (\{x_1\}, \{x_2\}, \{x_3\}, \{x_4\}, \{x_5\})$ be a $k$-clustering of $X$. For two choices of a different $k$-clustering $\C'$ of $X$, Figures \ref{fig:cyclical} and \ref{fig:sequential} give the corresponding clustering-difference graphs which represent a cyclical and sequential move of items.}\label{fig:cdgs}
\end{figure}
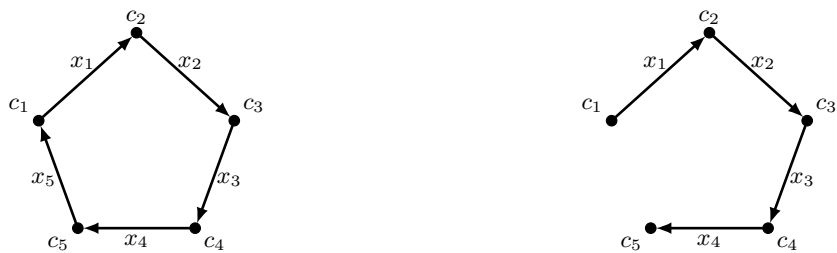

The clustering-difference graph plays an important role in the analysis of bounded-size and fixed-size partition polytopes. For the fixed-size partition polytope, a $CDG$ decomposes into directed cycles since all vertices in the graph must have equal indegree and outdegree. Two vertices in the polytope then share an edge if and only if the corresponding $CDG$ consists of a single directed cycle. This characterization has been used to construct edge walks between vertices of the polytope using sequences of cyclical moves, resulting in upper bounds on the combinatorial diameter \cite{b-13}.

Although the edges of the bounded-size partition polytope have a more complicated characterization, its \textit{circuits} analogously correspond to clustering-difference graphs consisting of either a single directed cycle or a single directed path \cite{bv-17}. Hence, devising sequences of cyclical and sequential moves for transforming $\C$ into $\C'$ can be interpreted as constructing a \textit{circuit walk} between the corresponding vertices in the polytope. As we will show in Sections \ref{sec:bounds} and \ref{sec:diameter}, this allows us to bound the \textit{circuit distance} between vertices and the related \textit{circuit diameter} of the bounded-size partition polytope.

Therefore, we are interested in constructing a transformation from $\C$ to $\C'$ using as few of these cyclical and sequential moves as possible. We call the minimum number of required moves the \textit{transformation distance} $d(\C, \C')$ from $\C$ to $\C'$.

\begin{definition}[Transformation Distance]
For $k$-clusterings $\C, \C'$ of the same data set, the \textit{transformation distance} $d(\C, \C')$ is the minimum number of cyclical and sequential moves needed to transform $\C$ into $\C'$.
\end{definition}

\noindent Hence, $d(\C, \C')$ is a relaxation of the circuit distance between the corresponding vertices of the bounded-size partition polytope -- as long as no cluster size constraints are violated during a sequence of moves used to achieve $d(\C, \C')$, the two distances are equal. 

A naive approach for bounding $d(\C, \C')$ is to decompose $\CDG$ into paths and cycles and then simply apply the corresponding moves individually to perform the clustering transformation. However, even when such a decomposition is optimal, $d(\C, \C')$ can be significantly less than the number of parts in the decomposition. For example, Figure~\ref{fig:disjoint_cycles:orig_CDG} depicts a case in which $\CDG$ consists of four vertex-disjoint cycles, which trivially implies $d(\C, \C') \leq 4$. Nevertheless, due to the fact that any permutation can be expressed as the product of two cyclic permutations \cite{br-74}, only two cyclical moves are needed to transform $\C$ into $\C'$, implying $d(\C, \C') = 2$. For simplicity, we will use the term \textit{disjoint} in place of \textit{vertex-disjoint} throughout the remainder of the paper. Whenever two components of a $CDG$ are not vertex-disjoint, we will say that they \textit{intersect}.

\begin{proposition}[\cite{br-74}, Lemma 7 in \cite{b-13}]\label{prop:disjoint_cycles}
Let $\C, \C'$ be $k$-clusterings for which $\CDG$ decomposes into disjoint cycles. Then $d(\C, \C') \leq 2$. 
\end{proposition}

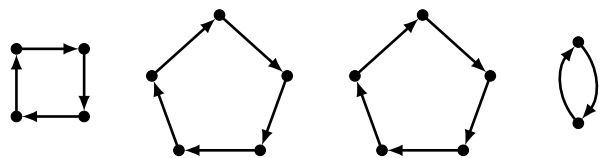
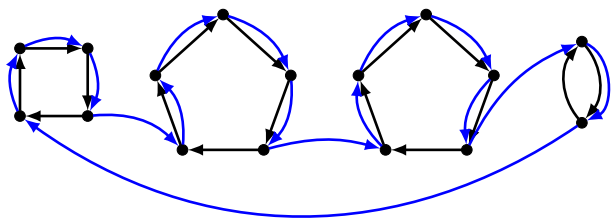
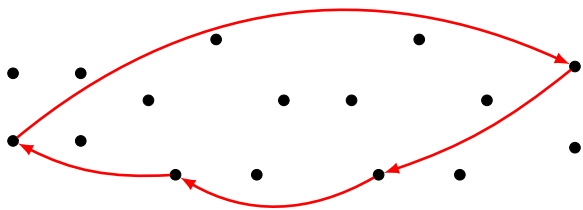
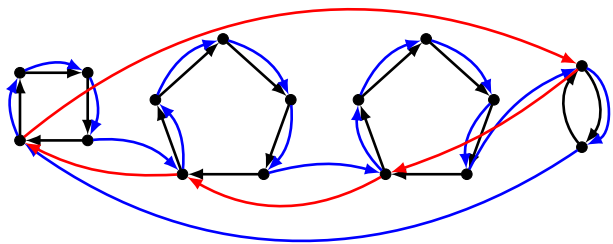
\begin{figure}
\centering

\begin{subfigure}{0.8\textwidth}
\centering
\begin{tikzpicture}[vertices/.style={draw, fill=black, circle, inner sep=0pt, minimum size = 4pt, outer sep=0pt}, scale=0.9]

\node[vertices] (w_1) at (0,0.5) {};
\node[vertices] (w_2) at (0,1.5) {};
\node[vertices] (w_3) at (1, 1.5) {};
\node[vertices] (w_4) at (1, 0.5) {};

\node[vertices] (c_1) at (2.0,1.1) {};
\node[vertices] (c_2) at (3.0, 2) {};
\node[vertices] (c_3) at (4.0, 1.1) {};
\node[vertices] (c_4) at (3.6, 0) {};
\node[vertices] (c_5) at (2.4, 0) {};

\node[vertices] (d_1) at (5,1.1) {};
\node[vertices] (d_2) at (6, 2) {};
\node[vertices] (d_3) at (7, 1.1) {};
\node[vertices] (d_4) at (6.6, 0) {};
\node[vertices] (d_5) at (5.4, 0) {};

\node[vertices] (z_1) at (8.3,0.4) {};
\node[vertices] (z_2) at (8.3,1.6) {};

\foreach \to/\from in {w_1/w_2, w_2/w_3, w_3/w_4, w_4/w_1,
	c_1/c_2, c_2/c_3, c_3/c_4, c_4/c_5, c_5/c_1,
	d_1/d_2, d_2/d_3, d_3/d_4, d_4/d_5, d_5/d_1}
\draw[draw=black, line width= 1,  ->, >=latex]  (\to)--(\from);

\path [draw=black,line width= 1,->,>=latex,black] (z_1) edge[bend left=40] (z_2);
\path [draw=black,line width= 1,->,>=latex,black] (z_2) edge[bend left=40] (z_1);


\end{tikzpicture}
\caption{The original clustering-difference graph $CDG(\C, \C')$.}
\label{fig:disjoint_cycles:orig_CDG}
\end{subfigure}

\vspace{.1in}
\begin{subfigure}{0.8 \textwidth}
\centering
\begin{tikzpicture}[vertices/.style={draw, fill=black, circle, inner sep=0pt, minimum size = 4pt, outer sep=0pt}, scale=0.9]

\node[vertices] (w_1) at (0,0.5) {};
\node[vertices] (w_2) at (0,1.5) {};
\node[vertices] (w_3) at (1, 1.5) {};
\node[vertices] (w_4) at (1, 0.5) {};

\node[vertices] (c_1) at (2.0,1.1) {};
\node[vertices] (c_2) at (3.0, 2) {};
\node[vertices] (c_3) at (4.0, 1.1) {};
\node[vertices] (c_4) at (3.6, 0) {};
\node[vertices] (c_5) at (2.4, 0) {};

\node[vertices] (d_1) at (5,1.1) {};
\node[vertices] (d_2) at (6, 2) {};
\node[vertices] (d_3) at (7, 1.1) {};
\node[vertices] (d_4) at (6.6, 0) {};
\node[vertices] (d_5) at (5.4, 0) {};

\node[vertices] (z_1) at (8.3,0.4) {};
\node[vertices] (z_2) at (8.3,1.6) {};

\foreach \to/\from in {w_1/w_2, w_2/w_3, w_3/w_4, w_4/w_1,
	c_1/c_2, c_2/c_3, c_3/c_4, c_4/c_5, c_5/c_1,
	d_1/d_2, d_2/d_3, d_3/d_4, d_4/d_5, d_5/d_1}
\draw[draw=black, line width= 1,  ->, >=latex]  (\to)--(\from);

\path [draw=black,line width= 1,->,>=latex,black] (z_1) edge[bend left=40] (z_2);
\path [draw=black,line width= 1,->,>=latex,black] (z_2) edge[bend left=40] (z_1);


\path [draw=blue,line width= 1,->,>=latex,blue] (w_1) edge[bend left=25] (w_2);
\path [draw=blue,line width= 1,->,>=latex,blue] (w_2) edge[bend left=25] (w_3);
\path [draw=blue,line width= 1,->,>=latex,blue] (w_3) edge[bend left=25] (w_4);
\path [draw=blue,line width= 1,->,>=latex,blue] (w_4) edge[bend left=25] (c_5);

\path [draw=blue,line width= 1,->,>=latex,blue] (c_5) edge[bend right=25] (c_1);
\path [draw=blue,line width= 1,->,>=latex,blue] (c_1) edge[bend left=25] (c_2);
\path [draw=blue,line width= 1,->,>=latex,blue] (c_2) edge[bend left=25] (c_3);
\path [draw=blue,line width= 1,->,>=latex,blue] (c_3) edge[bend left=25] (c_4);
\path [draw=blue,line width= 1,->,>=latex,blue] (c_4) edge[bend left=15] (d_5);
\path [draw=blue,line width= 1,->,>=latex,blue] (d_5) edge[bend left=25] (d_1);
\path [draw=blue,line width= 1,->,>=latex,blue] (d_1) edge[bend left=25] (d_2);
\path [draw=blue,line width= 1,->,>=latex,blue] (d_2) edge[bend left=25] (d_3);
\path [draw=blue,line width= 1,->,>=latex,blue] (d_3) edge[bend right=25] (d_4);
\path [draw=blue,line width= 1,->,>=latex,blue] (d_4) edge[bend left=20] (z_2);
\path [draw=blue,line width= 1,->,>=latex,blue] (z_2) edge[bend left=65] (z_1);
\path [draw=blue,line width= 1,->,>=latex,blue] (z_1) edge[bend left=35] (w_1);

\end{tikzpicture}
\caption{The first cyclical move, given by the blue edges.}\label{fig:disjoint_cycles:first_exchange}
\end{subfigure}

\vspace{.1in}
\begin{subfigure}{0.8 \textwidth}
\centering
\begin{tikzpicture}[vertices/.style={draw, fill=black, circle, inner sep=0pt, minimum size = 4pt, outer sep=0pt}, scale=0.9]

\node[vertices] (w_1) at (0,0.5) {};
\node[vertices] (w_2) at (0,1.5) {};
\node[vertices] (w_3) at (1, 1.5) {};
\node[vertices] (w_4) at (1, 0.5) {};

\node[vertices] (c_1) at (2.0,1.1) {};
\node[vertices] (c_2) at (3.0, 2) {};
\node[vertices] (c_3) at (4.0, 1.1) {};
\node[vertices] (c_4) at (3.6, 0) {};
\node[vertices] (c_5) at (2.4, 0) {};

\node[vertices] (d_1) at (5,1.1) {};
\node[vertices] (d_2) at (6, 2) {};
\node[vertices] (d_3) at (7, 1.1) {};
\node[vertices] (d_4) at (6.6, 0) {};
\node[vertices] (d_5) at (5.4, 0) {};

\node[vertices] (z_1) at (8.3,0.4) {};
\node[vertices] (z_2) at (8.3,1.6) {};


\path [draw=red,line width= 1,->,>=latex,red] (z_2) edge[bend left=10] (d_5);
\path [draw=red,line width= 1,->,>=latex,red] (d_5) edge[bend left=30] (c_5);
\path [draw=red,line width= 1,->,>=latex,red] (c_5) edge[bend left=15] (w_1);
\path [draw=red,line width= 1,->,>=latex,red] (w_1) edge[bend left=32] (z_2);

\end{tikzpicture}
\caption{The resulting clustering-difference graph after the first cyclical move is applied (and hence, the second cyclical move in the double-move), given by the red edges.}\label{fig:disjoint_cycles:second_exchange}
\end{subfigure}

\vspace{.1in}
\begin{subfigure}{0.8 \textwidth}
\centering
\begin{tikzpicture}[vertices/.style={draw, fill=black, circle, inner sep=0pt, minimum size = 4pt, outer sep=0pt}, scale=0.9]

\node[vertices] (w_1) at (0,0.5) {};
\node[vertices] (w_2) at (0,1.5) {};
\node[vertices] (w_3) at (1, 1.5) {};
\node[vertices] (w_4) at (1, 0.5) {};

\node[vertices] (c_1) at (2.0,1.1) {};
\node[vertices] (c_2) at (3.0, 2) {};
\node[vertices] (c_3) at (4.0, 1.1) {};
\node[vertices] (c_4) at (3.6, 0) {};
\node[vertices] (c_5) at (2.4, 0) {};

\node[vertices] (d_1) at (5,1.1) {};
\node[vertices] (d_2) at (6, 2) {};
\node[vertices] (d_3) at (7, 1.1) {};
\node[vertices] (d_4) at (6.6, 0) {};
\node[vertices] (d_5) at (5.4, 0) {};

\node[vertices] (z_1) at (8.3,0.4) {};
\node[vertices] (z_2) at (8.3,1.6) {};

\foreach \to/\from in {w_1/w_2, w_2/w_3, w_3/w_4, w_4/w_1,
	c_1/c_2, c_2/c_3, c_3/c_4, c_4/c_5, c_5/c_1,
	d_1/d_2, d_2/d_3, d_3/d_4, d_4/d_5, d_5/d_1}
\draw[draw=black, line width= 1,  ->, >=latex]  (\to)--(\from);

\path [draw=black,line width= 1,->,>=latex,black] (z_1) edge[bend left=40] (z_2);
\path [draw=black,line width= 1,->,>=latex,black] (z_2) edge[bend left=40] (z_1);


\path [draw=blue,line width= 1,->,>=latex,blue] (w_1) edge[bend left=25] (w_2);
\path [draw=blue,line width= 1,->,>=latex,blue] (w_2) edge[bend left=25] (w_3);
\path [draw=blue,line width= 1,->,>=latex,blue] (w_3) edge[bend left=25] (w_4);
\path [draw=blue,line width= 1,->,>=latex,blue] (w_4) edge[bend left=25] (c_5);

\path [draw=blue,line width= 1,->,>=latex,blue] (c_5) edge[bend right=25] (c_1);
\path [draw=blue,line width= 1,->,>=latex,blue] (c_1) edge[bend left=25] (c_2);
\path [draw=blue,line width= 1,->,>=latex,blue] (c_2) edge[bend left=25] (c_3);
\path [draw=blue,line width= 1,->,>=latex,blue] (c_3) edge[bend left=25] (c_4);
\path [draw=blue,line width= 1,->,>=latex,blue] (c_4) edge[bend left=15] (d_5);
\path [draw=blue,line width= 1,->,>=latex,blue] (d_5) edge[bend left=25] (d_1);
\path [draw=blue,line width= 1,->,>=latex,blue] (d_1) edge[bend left=25] (d_2);
\path [draw=blue,line width= 1,->,>=latex,blue] (d_2) edge[bend left=25] (d_3);
\path [draw=blue,line width= 1,->,>=latex,blue] (d_3) edge[bend right=25] (d_4);
\path [draw=blue,line width= 1,->,>=latex,blue] (d_4) edge[bend left=20] (z_2);
\path [draw=blue,line width= 1,->,>=latex,blue] (z_2) edge[bend left=65] (z_1);
\path [draw=blue,line width= 1,->,>=latex,blue] (z_1) edge[bend left=35] (w_1);

\path [draw=red,line width= 1,->,>=latex,red] (z_2) edge[bend left=10] (d_5);
\path [draw=red,line width= 1,->,>=latex,red] (d_5) edge[bend left=30] (c_5);
\path [draw=red,line width= 1,->,>=latex,red] (c_5) edge[bend left=15] (w_1);
\path [draw=red,line width= 1,->,>=latex,red] (w_1) edge[bend left=32] (z_2);

\end{tikzpicture}
\caption{A combined visualization of the double-move for transforming $\C$ into $\C'$.}\label{fig:disjoint_cycles:both_exchanges}
\end{subfigure}

\caption{A case where $\CDG$ consists of four disjoint cycles. The above double-move from \cite{b-13} can be used to transform $\C$ into $\C'$ in only two cyclical moves, implying $d(\C, \C') = 2$.}\label{fig:disjoint_cycles}
\end{figure}

In \cite{b-13}, Proposition~\ref{prop:disjoint_cycles} is used to integrate disjoint cyclical moves into what we will call a \textit{double-move}: a sequence of two moves which results in the desired changes to the underlying clustering-difference graph. See Figure~\ref{fig:disjoint_cycles} for a visualization of this double-move. In a first cyclical move of items, depicted in Figure~\ref{fig:disjoint_cycles:first_exchange}, the transfers corresponding to all but one edge from each cycle are correctly applied -- the items corresponding to the remaining edges are temporarily sent to incorrect destinations across the cycles. However, a second cyclical move, depicted in Figure~\ref{fig:disjoint_cycles:second_exchange}, can then be used to send each of these misplaced items to its correct destination, completing the clustering transformation.

In the following lemmas and the upcoming Theorems~\ref{thm:two_exchanges} and \ref{thm:three_intersections}, we show how similar double-moves can be used to integrate both cyclical and sequential moves when transforming general $k$-clusterings. This becomes more technically challenging due to the different natures of the two types of moves -- cyclical moves do not alter the sizes of any underlying cluster while sequential moves alter the sizes of exactly two clusters by one. First, generalizing Proposition~\ref{prop:disjoint_cycles}, we show how to integrate sets of disjoint cycles and paths from a $CDG$. Note that the bound on $d(\C, \C')$ in the following lemma depends only on the number of paths in the decomposition and not on the number of cycles.

\begin{lemma}\label{lem:disjoint_paths}
Let $\C, \C'$ be $k$-clusterings for which $CDG(\C, \C')$ decomposes into disjoint cycles and paths. Then $d(\C, \C') \leq \max\{2, \ t\}$, where $t$ is the number of paths in the decomposition.
\end{lemma}
\begin{proof}
Let $t$ denote the number of directed paths and $s$ the number of directed cycles in the decomposition of $\CDG$. When $t = 0$, the result follows from Proposition~\ref{prop:disjoint_cycles}, so assume $t \geq 1$. We may also assume $s \geq 1$, else the $t$ sequential moves could simply be applied individually.

First suppose $t=2$. Let $P_1, P_2$ denote the paths of $\CDG$ and let $Y_1,...,Y_s$ denote the cycles. For $i = 1,...,s$, select any edge $e_i := (u_i, v_i)$ from cycle $Y_i$. In addition, let $p_1$ denote the tail of $P_1$, $p_2$ the tail of $P_2$, and $w$ the neighbor of $p_1$ on $P_1$. See Figure~\ref{fig:two_paths:orig_CDG}. We will apply a double-move consisting of two sequential moves to transform clustering $\C$ into $\C'$. 

For the first sequential move, first introduce an edge from $p_1$ to $v_1$, sending to $v_1$ the item from $p_1$ intended for $w$. After traveling along this edge, follow $Y_1 - e_1$ from $v_1$ to $u_1$. Next, if $s > 1$, introduce an edge from $u_1$ to $v_2$, whose item associated with $e_1$, and then travel along $Y_2 - e_2$ to $u_2$. Repeat for $i = 2,...,s-1$ until $u_s$ is reached. Complete the move by introducing an edge from $u_s$ to $p_2$, whose item is associated with $e_s$, and then traveling along path $P_2$. See the blue edges in Figure~\ref{fig:two_paths:first_exchange} for a visualization of this move. 

Note that this first sequential move applies all transfers given by the cycles $Y_1,...,Y_s$ in $CDG(\C, \C')$ with the exception of those corresponding to the edges $e_i$ -- the item from $v_i$ intended for $u_i$ is temporarily sent to the wrong cluster. The move also applies all transfers given by $P_2$ but does not correctly apply any transfer from $P_1$. However, all misplaced items can be corrected and all remaining transfers can be applied via a single, second sequential move as seen in Figure~\ref{fig:two_paths:second_exchange}. First, send from $p_2$ to $v_s$ the item $p_2$ received from $u_s$. Next, for $i=s,...,2$, send from $v_i$ to $v_{i-1}$ the item $v_i$ received from $u_{i-1}$. Once $v_1$ is reached, send from $v_1$ to $w$ the item $v_1$ received from $p_1$. Finish the move by following the remaining edges of $P_1$.  After this second sequential move is applied, the transformation from $\C$ to $\C'$ is complete -- it follows that $d(\C, \C') = 2$. See Figure~\ref{fig:two_paths:both_exchanges} for a visualization of these moves integrated together into a double-move

\begin{figure}
\centering

\begin{subfigure}{0.8\textwidth}
\centering
\begin{tikzpicture}[vertices/.style={draw, fill=black, circle, inner sep=0pt, minimum size = 4pt, outer sep=0pt}]

\node[vertices] (w_1) at (0,0) {};
\node[vertices] (w_2) at (0,1) {};
\node[vertices] (w_3) at (0, 2) {};

\node[vertices] (c_1) at (1.5,1.1) {};
\node[vertices] (c_2) at (2.5, 2) {};
\node[vertices] (c_3) at (3.5, 1.1) {};
\node[vertices] (c_4) at (3.1, 0) {};
\node[vertices] (c_5) at (1.9, 0) {};

\node[vertices] (d_1) at (5,1.1) {};
\node[vertices] (d_2) at (6, 2) {};
\node[vertices] (d_3) at (7, 1.1) {};
\node[vertices] (d_4) at (6.6, 0) {};
\node[vertices] (d_5) at (5.4, 0) {};

\node[vertices] (z_1) at (8.5,0) {};
\node[vertices] (z_2) at (8.5,1) {};
\node[vertices] (z_3) at (8.5, 2) {};

\foreach \to/\from in {w_1/w_2, w_2/w_3,
	c_1/c_2, c_2/c_3, c_3/c_4, c_4/c_5, c_5/c_1,
	d_1/d_2, d_2/d_3, d_3/d_4, d_4/d_5, d_5/d_1,
	z_1/z_2, z_2/z_3}
\draw[draw=black, line width= 1,  ->, >=latex]  (\to)--(\from);

\node[below left] at (w_1) {$p_1$};
\node[left] at (w_2) {$w$};
\node[below right] at (z_1) {$p_2$};
\node[below right] at (c_4) {$u_1$};
\node[below left] at (c_5) {$v_1$};
\node[below right] at (d_4) {$u_2$};
\node[below left] at (d_5) {$v_2$};

\end{tikzpicture}
\caption{The original clustering-difference graph $CDG(\C, \C')$.}
\label{fig:two_paths:orig_CDG}
\end{subfigure}

\vspace{.1in}
\begin{subfigure}{0.8 \textwidth}
\centering
\begin{tikzpicture}[vertices/.style={draw, fill=black, circle, inner sep=0pt, minimum size = 4pt, outer sep=0pt}]

\node[vertices] (w_1) at (0,0) {};
\node[vertices] (w_2) at (0,1) {};
\node[vertices] (w_3) at (0, 2) {};

\node[vertices] (c_1) at (1.5,1.1) {};
\node[vertices] (c_2) at (2.5, 2) {};
\node[vertices] (c_3) at (3.5, 1.1) {};
\node[vertices] (c_4) at (3.1, 0) {};
\node[vertices] (c_5) at (1.9, 0) {};

\node[vertices] (d_1) at (5,1.1) {};
\node[vertices] (d_2) at (6, 2) {};
\node[vertices] (d_3) at (7, 1.1) {};
\node[vertices] (d_4) at (6.6, 0) {};
\node[vertices] (d_5) at (5.4, 0) {};

\node[vertices] (z_1) at (8.5,0) {};
\node[vertices] (z_2) at (8.5,1) {};
\node[vertices] (z_3) at (8.5, 2) {};

\foreach \to/\from in {w_1/w_2, w_2/w_3,
	c_1/c_2, c_2/c_3, c_3/c_4, c_4/c_5, c_5/c_1,
	d_1/d_2, d_2/d_3, d_3/d_4, d_4/d_5, d_5/d_1,
	z_1/z_2, z_2/z_3}
\draw[draw=black, line width= 1,  ->, >=latex]  (\to)--(\from);

\node[below left] at (w_1) {$p_1$};
\node[left] at (w_2) {$w$};
\node[below right] at (z_1) {$p_2$};
\node[below right] at (c_4) {$u_1$};
\node[below left] at (c_5) {$v_1$};
\node[below right] at (d_4) {$u_2$};
\node[below left] at (d_5) {$v_2$};

\path [draw=blue,line width= 1,->,>=latex,blue] (w_1) edge[bend left=10] (c_5);
\path [draw=blue,line width= 1,->,>=latex,blue] (c_5) edge[bend left=25] (c_1);
\path [draw=blue,line width= 1,->,>=latex,blue] (c_1) edge[bend left=25] (c_2);
\path [draw=blue,line width= 1,->,>=latex,blue] (c_2) edge[bend left=25] (c_3);
\path [draw=blue,line width= 1,->,>=latex,blue] (c_3) edge[bend left=25] (c_4);
\path [draw=blue,line width= 1,->,>=latex,blue] (c_4) edge[bend left=15] (d_5);
\path [draw=blue,line width= 1,->,>=latex,blue] (d_5) edge[bend left=25] (d_1);
\path [draw=blue,line width= 1,->,>=latex,blue] (d_1) edge[bend left=25] (d_2);
\path [draw=blue,line width= 1,->,>=latex,blue] (d_2) edge[bend left=25] (d_3);
\path [draw=blue,line width= 1,->,>=latex,blue] (d_3) edge[bend left=25] (d_4);
\path [draw=blue,line width= 1,->,>=latex,blue] (d_4) edge[bend left=15] (z_1);
\path [draw=blue,line width= 1,->,>=latex,blue] (z_1) edge[bend right=30] (z_2);
\path [draw=blue,line width= 1,->,>=latex,blue] (z_2) edge[bend right=30] (z_3);

\end{tikzpicture}
\caption{The first sequential move.}\label{fig:two_paths:first_exchange}
\end{subfigure}

\vspace{.1in}
\begin{subfigure}{0.8 \textwidth}
\centering
\begin{tikzpicture}[vertices/.style={draw, fill=black, circle, inner sep=0pt, minimum size = 4pt, outer sep=0pt}]

\node[vertices] (w_1) at (0,0) {};
\node[vertices] (w_2) at (0,1) {};
\node[vertices] (w_3) at (0, 2) {};

\node[vertices] (c_1) at (1.5,1.1) {};
\node[vertices] (c_2) at (2.5, 2) {};
\node[vertices] (c_3) at (3.5, 1.1) {};
\node[vertices] (c_4) at (3.1, 0) {};
\node[vertices] (c_5) at (1.9, 0) {};

\node[vertices] (d_1) at (5,1.1) {};
\node[vertices] (d_2) at (6, 2) {};
\node[vertices] (d_3) at (7, 1.1) {};
\node[vertices] (d_4) at (6.6, 0) {};
\node[vertices] (d_5) at (5.4, 0) {};

\node[vertices] (z_1) at (8.5,0) {};
\node[vertices] (z_2) at (8.5,1) {};
\node[vertices] (z_3) at (8.5, 2) {};

\foreach \to/\from in {w_2/w_3}
\draw[draw=black, line width= 1,  ->, >=latex]  (\to)--(\from);

\node[below left] at (w_1) {$p_1$};
\node[left] at (w_2) {$w$};
\node[below right] at (z_1) {$p_2$};
\node[below right] at (c_4) {$u_1$};
\node[below left] at (c_5) {$v_1$};
\node[below right] at (d_4) {$u_2$};
\node[below] at (d_5) {$v_2$};

\path [draw=red,line width= 1,->,>=latex,red] (z_1) edge[bend left=30] (d_5);
\path [draw=red,line width= 1,->,>=latex,red] (d_5) edge[bend left=30] (c_5);
\path [draw=red,line width= 1,->,>=latex,red] (c_5) edge[bend right=15] (w_2);

\end{tikzpicture}
\caption{The resulting $CDG$ after the first sequential move is applied. This directed path corresponds to the second sequential move used to transform $\C$ into $\C'$.}\label{fig:two_paths:second_exchange}
\end{subfigure}

\vspace{.1in}
\begin{subfigure}{0.8 \textwidth}
\centering
\begin{tikzpicture}[vertices/.style={draw, fill=black, circle, inner sep=0pt, minimum size = 4pt, outer sep=0pt}]

\node[vertices] (w_1) at (0,0) {};
\node[vertices] (w_2) at (0,1) {};
\node[vertices] (w_3) at (0, 2) {};

\node[vertices] (c_1) at (1.5,1.1) {};
\node[vertices] (c_2) at (2.5, 2) {};
\node[vertices] (c_3) at (3.5, 1.1) {};
\node[vertices] (c_4) at (3.1, 0) {};
\node[vertices] (c_5) at (1.9, 0) {};

\node[vertices] (d_1) at (5,1.1) {};
\node[vertices] (d_2) at (6, 2) {};
\node[vertices] (d_3) at (7, 1.1) {};
\node[vertices] (d_4) at (6.6, 0) {};
\node[vertices] (d_5) at (5.4, 0) {};

\node[vertices] (z_1) at (8.5,0) {};
\node[vertices] (z_2) at (8.5,1) {};
\node[vertices] (z_3) at (8.5, 2) {};

\foreach \to/\from in {w_1/w_2, w_2/w_3,
	c_1/c_2, c_2/c_3, c_3/c_4, c_4/c_5, c_5/c_1,
	d_1/d_2, d_2/d_3, d_3/d_4, d_4/d_5, d_5/d_1,
	z_1/z_2, z_2/z_3}
\draw[draw=black, line width= 1,  ->, >=latex]  (\to)--(\from);

\node[below left] at (w_1) {$p_1$};
\node[left] at (w_2) {$w$};
\node[below right] at (z_1) {$p_2$};
\node[below right] at (c_4) {$u_1$};
\node[below left] at (c_5) {$v_1$};
\node[below right] at (d_4) {$u_2$};
\node[below] at (d_5) {$v_2$};

\path [draw=blue,line width= 1,->,>=latex,blue] (w_1) edge[bend left=10] (c_5);
\path [draw=blue,line width= 1,->,>=latex,blue] (c_5) edge[bend left=20] (c_1);
\path [draw=blue,line width= 1,->,>=latex,blue] (c_1) edge[bend left=25] (c_2);
\path [draw=blue,line width= 1,->,>=latex,blue] (c_2) edge[bend left=25] (c_3);
\path [draw=blue,line width= 1,->,>=latex,blue] (c_3) edge[bend left=25] (c_4);
\path [draw=blue,line width= 1,->,>=latex,blue] (c_4) edge[bend left=15] (d_5);
\path [draw=blue,line width= 1,->,>=latex,blue] (d_5) edge[bend left=25] (d_1);
\path [draw=blue,line width= 1,->,>=latex,blue] (d_1) edge[bend left=25] (d_2);
\path [draw=blue,line width= 1,->,>=latex,blue] (d_2) edge[bend left=25] (d_3);
\path [draw=blue,line width= 1,->,>=latex,blue] (d_3) edge[bend left=25] (d_4);
\path [draw=blue,line width= 1,->,>=latex,blue] (d_4) edge[bend left=15] (z_1);
\path [draw=blue,line width= 1,->,>=latex,blue] (z_1) edge[bend right=30] (z_2);
\path [draw=blue,line width= 1,->,>=latex,blue] (z_2) edge[bend right=30] (z_3);

\path [draw=red,line width= 1,->,>=latex,red] (z_1) edge[bend left=30] (d_5);
\path [draw=red,line width= 1,->,>=latex,red] (d_5) edge[bend left=30] (c_5);
\path [draw=red,line width= 1,->,>=latex,red] (c_5) edge[bend right=15] (w_2);
\path [draw=red,line width= 1,->,>=latex,red] (w_2) edge[bend right=30] (w_3);

\end{tikzpicture}
\caption{A combined visualization of the double-move for transforming $\C$ into $\C'$.}\label{fig:two_paths:both_exchanges}
\end{subfigure}

\caption{A case where $CDG(\C, \C')$ consists of two paths and multiple cycles. Independent of the number of cycles, only two sequential moves are needed to transform $\C$ into $\C'$, so $d(\C, \C') = 2$.}\label{fig:two_paths}
\end{figure}
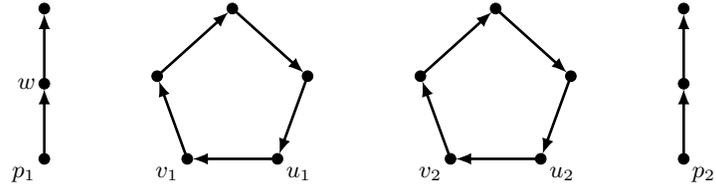
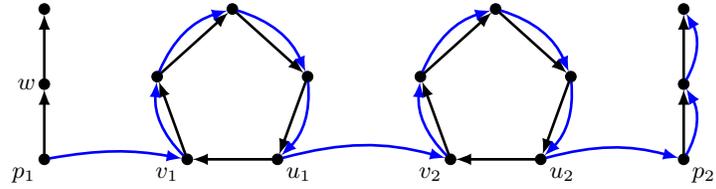
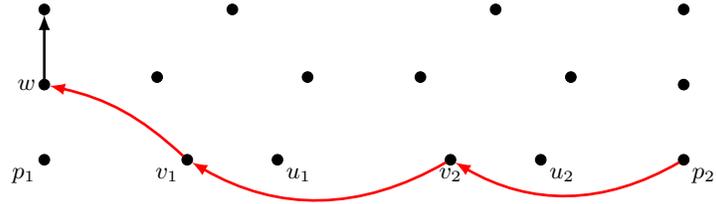
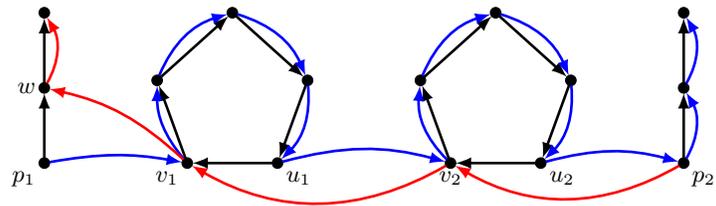

Next suppose $t > 2$. We can apply the double-move from the previous case to remove all cycles and any two paths from $\CDG$. The remaining $t-2$ paths can then be applied via individual sequential moves. Hence, $d(\C, \C') \leq t$.

Lastly, suppose $t = 1$. Again let $P_1$ denote the single path of $CDG(\C, \C')$, let $p_1$ denote the tail of $P_1$, let $w$ denote the neighbor of $p_1$ on $P_1$, and choose an edge $e_i := (u_i, v_i)$ from $Y_i$ for $i=1,...,s$. We will apply a sequential move followed by a cyclical move to transform $\C$ into $\C'$. See Figure~\ref{fig:one_path} for a visualization of the double-move.

For the first move, as in the previous case, first introduce an edge from $p_1$ to $v_1$ and then travel along $Y_1 - e_1$ from $v_1$ to $u_1$. If $s > 1$, then for $i = 1,...,s-1$, introduce an edge from $u_i$ to $v_i$ and follow $Y_i - e_i$ to $u_i$ until $u_s$ is reached. Complete the sequential move by introducing an edge from $u_s$ to $w$ and then following the remaining edges of $P_1$.

The second cyclical move corrects all items sent to incorrect destinations by the first move. Namely, first follow the edge from $w$ to $v_s$, sending to $v_s$ the item $w$ received from $u_s$. Next, for $i = s,...,2$, follow the edges from $v_{i}$ to $v_{i-1}$ until $v_1$ is reached. Complete the move by following the edge from $v_1$ to $w$, sending $w$ the item $v_1$ received from $p_1$. This completes the transformation from $\C$ to $\C'$, so $d(\C, \C') \leq 2$. 

Therefore, in each case, $d(\C, \C') \leq \max \{ 2, \  t \}$. \eoproof
\end{proof}

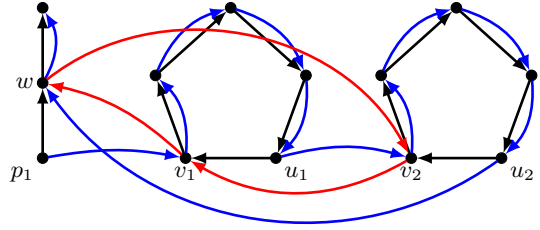
\begin{figure}
\centering
\begin{tikzpicture}[vertices/.style={draw, fill=black, circle, inner sep=0pt, minimum size = 4pt, outer sep=0pt}]

\node[vertices] (w_1) at (0,0) {};
\node[vertices] (w_2) at (0,1) {};
\node[vertices] (w_3) at (0, 2) {};

\node[vertices] (c_1) at (1.5,1.1) {};
\node[vertices] (c_2) at (2.5, 2) {};
\node[vertices] (c_3) at (3.5, 1.1) {};
\node[vertices] (c_4) at (3.1, 0) {};
\node[vertices] (c_5) at (1.9, 0) {};

\node[vertices] (d_1) at (4.5,1.1) {};
\node[vertices] (d_2) at (5.5, 2) {};
\node[vertices] (d_3) at (6.5, 1.1) {};
\node[vertices] (d_4) at (6.1, 0) {};
\node[vertices] (d_5) at (4.9, 0) {};

\foreach \to/\from in {w_1/w_2, w_2/w_3,
	c_1/c_2, c_2/c_3, c_3/c_4, c_4/c_5, c_5/c_1,
	d_1/d_2, d_2/d_3, d_3/d_4, d_4/d_5, d_5/d_1}
\draw[draw=black, line width= 1,  ->, >=latex]  (\to)--(\from);

\node[below left] at (w_1) {$p_1$};
\node[left] at (w_2) {$w$};
\node[below right] at (c_4) {$u_1$};
\node[below] at (c_5) {$v_1$};
\node[below right] at (d_4) {$u_2$};
\node[below] at (d_5) {$v_2$};

\path [draw=blue,line width= 1,->,>=latex,blue] (w_1) edge[bend left=10] (c_5);
\path [draw=blue,line width= 1,->,>=latex,blue] (c_5) edge[bend right=25] (c_1);
\path [draw=blue,line width= 1,->,>=latex,blue] (c_1) edge[bend left=25] (c_2);
\path [draw=blue,line width= 1,->,>=latex,blue] (c_2) edge[bend left=25] (c_3);
\path [draw=blue,line width= 1,->,>=latex,blue] (c_3) edge[bend left=25] (c_4);
\path [draw=blue,line width= 1,->,>=latex,blue] (c_4) edge[bend left=15] (d_5);
\path [draw=blue,line width= 1,->,>=latex,blue] (d_5) edge[bend right=25] (d_1);
\path [draw=blue,line width= 1,->,>=latex,blue] (d_1) edge[bend left=25] (d_2);
\path [draw=blue,line width= 1,->,>=latex,blue] (d_2) edge[bend left=25] (d_3);
\path [draw=blue,line width= 1,->,>=latex,blue] (d_3) edge[bend left=25] (d_4);
\path [draw=blue,line width= 1,->,>=latex,blue] (d_4) edge[bend left=45] (w_2);
\path [draw=blue,line width= 1,->,>=latex,blue] (w_2) edge[bend right=30] (w_3);

\path [draw=red,line width= 1,->,>=latex,red] (w_2) edge[bend left=50] (d_5);
\path [draw=red,line width= 1,->,>=latex,red] (d_5) edge[bend left=30] (c_5);
\path [draw=red,line width= 1,->,>=latex,red] (c_5) edge[bend right=15] (w_2);
\end{tikzpicture}
\caption{A visualization of the double-move for transforming $\C$ into $\C'$ when $CDG(\C, \C')$ consists of a single path and multiple cycles. The original $\CDG$ is given by the black edges, the first sequential move is given by the blue edges, and the second cyclical move is given by the red edges.}\label{fig:one_path}
\end{figure}

Even when paths and cycles are not completely disjoint, the corresponding moves may still be integrated together into a double-move. In the following lemma, we show that if a path intersects at most one cycle out of a collection of disjoint cycles, only two moves are required to apply all corresponding transfers.

\begin{lemma}\label{lem:one_cycle_intersect}
Let $\C, \C'$ be $k$-clusterings for which $CDG(\C, \C')$ consists of $s \geq 1$ disjoint cycles and a single path $P$, which intersects at most one of the cycles. Then $d(\C, \C') \leq 2$. 
\end{lemma}

\begin{proof}
Let $Y_1,...,Y_s$ denote the cycles of $CDG(\C, \C')$. If $P$ does not intersect any of the cycles, we can apply Lemma~\ref{lem:disjoint_paths}, so assume $P$ intersects $Y_1$. 
 We may also assume $s \geq 2$, else the sequential and cyclical moves corresponding to $P$ and $Y_1$ could simply be applied individually. Let $w_1,...,w_t$ denote (in order) the vertices of $P$, and select an edge $e_i := (u_i, v_i)$ from each cycle $Y_i$. Assume $e_1$ is chosen such that $v_1$ is the first vertex $w_j$ of $P$ that also belongs to $Y_1$. (Note that it is possible to have $j=1$ or $j = t$, but both of those cases are still covered by the following construction.) We will apply two sequential moves to transform $\C$ into $\C'$. See Figure~\ref{fig:one_intersect} for a visualization of $CGD(\C, \C')$ along with the double-move.

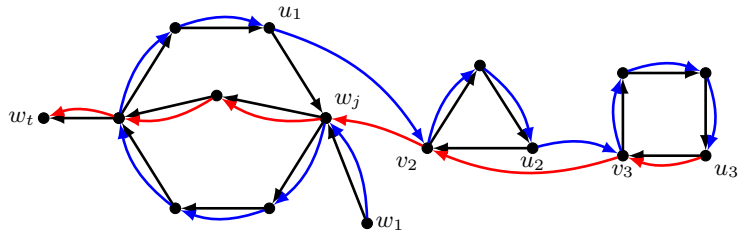
\begin{figure}
\centering
\begin{tikzpicture}[vertices/.style={draw, fill=black, circle, inner sep=0pt, minimum size = 4pt, outer sep=0pt}]

\node[vertices] (c_1) at (0,1.2) {};
\node[vertices] (c_2) at (0.75, 2.4) {};
\node[vertices] (c_3) at (2, 2.4) {};
\node[vertices] (c_4) at (2.75, 1.2) {};
\node[vertices] (c_5) at (2, 0) {};
\node[vertices] (c_6) at (0.75, 0) {};

\node[vertices] (w_1) at (3.3,-0.2) {};
\node[vertices] (w_2) at (2.75, 1.2) {};
\node[vertices] (w_3) at (1.3, 1.5) {};
\node[vertices] (w_4) at (0, 1.2) {};
\node[vertices] (w_5) at (-1, 1.2) {};

\node[vertices] (d_1) at (4.1,0.8) {};
\node[vertices] (d_2) at (4.8, 1.9) {};
\node[vertices] (d_3) at (5.5, 0.8) {};

\node[vertices] (e_1) at (6.7,0.7) {};
\node[vertices] (e_2) at (6.7, 1.8) {};
\node[vertices] (e_3) at (7.8, 1.8) {};
\node[vertices] (e_4) at (7.8, 0.7) {};

\foreach \to/\from in {
	w_1/w_2, w_2/w_3, w_3/w_4, w_4/w_5,
	c_1/c_2, c_2/c_3, c_3/c_4, c_4/c_5, c_5/c_6, c_6/c_1,
	d_1/d_2, d_2/d_3, d_3/d_1,
	e_1/e_2, e_2/e_3, e_3/e_4, e_4/e_1}
\draw[draw=black, line width= 1,  ->, >=latex]  (\to)--(\from);

\node[right] at (w_1) {$w_1$};
\node[above right] at (w_2) {$w_j$};
\node[left] at (w_5) {$w_t$};
\node[above right] at (c_3) {$u_1$};
\node[below left] at (d_1) {$v_2$};
\node[below] at (d_3) {$u_2$};
\node[below] at (e_1) {$v_3$};
\node[below right] at (e_4) {$u_3$};

\path [draw=blue,line width= 1,->,>=latex,blue] (w_1) edge[bend right=20] (w_2);
\path [draw=blue,line width= 1,->,>=latex,blue] (c_4) edge[bend left=20] (c_5);
\path [draw=blue,line width= 1,->,>=latex,blue] (c_5) edge[bend left=20] (c_6);
\path [draw=blue,line width= 1,->,>=latex,blue] (c_6) edge[bend left=20] (c_1);
\path [draw=blue,line width= 1,->,>=latex,blue] (c_1) edge[bend left=20] (c_2);
\path [draw=blue,line width= 1,->,>=latex,blue] (c_2) edge[bend left=20] (c_3);
\path [draw=blue,line width= 1,->,>=latex,blue] (c_3) edge[bend left=20] (d_1);
\path [draw=blue,line width= 1,->,>=latex,blue] (d_1) edge[bend left=20] (d_2);
\path [draw=blue,line width= 1,->,>=latex,blue] (d_2) edge[bend left=20] (d_3);
\path [draw=blue,line width= 1,->,>=latex,blue] (d_3) edge[bend left=20] (e_1);
\path [draw=blue,line width= 1,->,>=latex,blue] (e_1) edge[bend left=20] (e_2);
\path [draw=blue,line width= 1,->,>=latex,blue] (e_2) edge[bend left=20] (e_3);
\path [draw=blue,line width= 1,->,>=latex,blue] (e_3) edge[bend left=20] (e_4);

\path [draw=red,line width= 1,->,>=latex,red] (e_4) edge[bend left=20] (e_1);
\path [draw=red,line width= 1,->,>=latex,red] (e_1) edge[bend left=20] (d_1);
\path [draw=red,line width= 1,->,>=latex,red] (d_1) edge[bend right=10] (w_2);
\path [draw=red,line width= 1,->,>=latex,red] (w_2) edge[bend left=20] (w_3);
\path [draw=red,line width= 1,->,>=latex,red] (w_3) edge[bend left=20] (w_4);
\path [draw=red,line width= 1,->,>=latex,red] (w_4) edge[bend right=20] (w_5);

\end{tikzpicture}
\caption{A visualization of the double-move for transforming $\C$ into $\C'$ when $CDG(\C, \C')$ consists of multiple disjoint cycles and a path intersecting a single cycle.}\label{fig:one_intersect}
\end{figure}

For the first sequential move, start at $w_1$ and follow $P$ to $w_j = v_1$. Then follow the path formed by joining $Y_1 - e_1,...,Y_s - e_s$ with the introduced edges $(u_1, v_2),...,(u_{s-1}, v_s)$, terminating the move at $u_s$. Hence, the move reduces the size of the cluster corresponding to $w_1$, as desired, but it also increases the size of the cluster corresponding to $u_s$.

To correct this, apply a second sequential move starting at $u_s$ and terminating at $w_t$. First follow the edges $(u_s, v_s), (v_s, v_{s-1}),..., (v_2, v_1)$ to correct items misplaced across cycles. Next, since $v_1 = w_j$, follow $P$ along the vertices $w_j,...,w_t$ to complete the move. Since $P$ only intersects $Y_1$, no vertices are repeated in this second sequential move and it indeed corresponds to a single directed path in the $CDG$. All transfers corresponding to the original edges of $CDG(\C, \C')$ have then been correctly applied. \eoproof
\end{proof}

Note that in the double-move used for Lemma~\ref{lem:one_cycle_intersect}, the first sequential move temporarily increases the size of the cluster corresponding to $u_s$. For general $k$-clusterings this is not an issue, but for bounded-size $k$-clusterings this could potentially violate the upper bound on the size of the cluster. Unfortunately, this increase in cluster size is unavoidable for certain configurations of $\CDG$. Consider the example in Figure~\ref{fig:size_increase}. If only two moves are to be used to transform $\C$ into $\C'$, $c_1$ must correctly send an item to $c_2$ in each move. Hence, $c_1$ may not temporarily receive an item from $c_3$ or $c_4$ in the first move, and the cluster size corresponding to either $c_3$ or $c_4$ must be temporarily increased. Therefore, although $d(\C, \C') = 2$, the circuit distance between the corresponding vertices in a related bounded-size partition polytope may be 3 if the cluster sizes corresponding to $c_3$ and $c_4$ are already at their upper bounds. We further address this issue in Section~\ref{sec:diameter}.

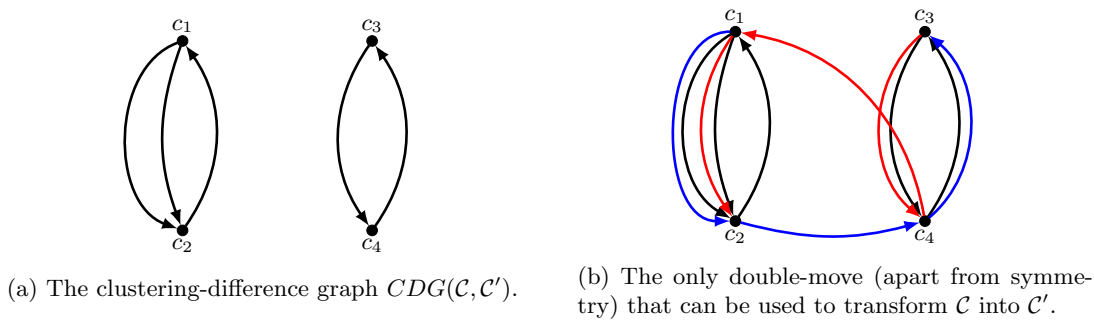
\begin{figure}
\centering
\begin{subfigure}{0.45 \textwidth}
\centering
\begin{tikzpicture}[vertices/.style={draw, fill=black, circle, inner sep=0pt, minimum size = 4pt, outer sep=0pt}, scale=1.4]

\node[vertices] (c_1) at (0,0) {};
\node[vertices] (c_2) at (0, 1.8) {};

\node[vertices] (c_3) at (1.8, 0) {};
\node[vertices] (c_4) at (1.8, 1.8) {};

\node[below] at (c_1) {$c_2$};
\node[above] at (c_2) {$c_1$};
\node[below] at (c_3) {$c_4$};
\node[above] at (c_4) {$c_3$};

\path [draw=black,line width= 1,->,>=latex] (c_1) edge[bend right=35] (c_2);
\path [draw=black,line width= 1,->,>=latex] (c_3) edge[bend right=35] (c_4);

\path [draw=black,line width= 1,->,>=latex] (c_2) edge[bend right=20] (c_1);
\path [draw=black,line width= 1,->,>=latex] (c_4) edge[bend right=35] (c_3);

\path [draw=black,line width= 1,->,>=latex] (c_2) edge[bend right=70] (c_1);

\end{tikzpicture}
\caption{The clustering-difference graph $\CDG$.}
\end{subfigure}
\qquad
\begin{subfigure}{0.45 \textwidth}
\centering
\begin{tikzpicture}[vertices/.style={draw, fill=black, circle, inner sep=0pt, minimum size = 4pt, outer sep=0pt}, scale=1.4]

\useasboundingbox (-0.25,-0.25) rectangle (2.15,2.15);

\node[vertices] (c_1) at (0,0) {};
\node[vertices] (c_2) at (0, 1.8) {};

\node[vertices] (c_3) at (1.8, 0) {};
\node[vertices] (c_4) at (1.8, 1.8) {};

\node[below] at (c_1) {$c_2$};
\node[above] at (c_2) {$c_1$};
\node[below] at (c_3) {$c_4$};
\node[above] at (c_4) {$c_3$};

\path [draw=black,line width= 1,->,>=latex] (c_1) edge[bend right=35] (c_2);
\path [draw=black,line width= 1,->,>=latex] (c_3) edge[bend right=35] (c_4);

\path [draw=black,line width= 1,->,>=latex] (c_2) edge[bend right=20] (c_1);
\path [draw=black,line width= 1,->,>=latex] (c_4) edge[bend right=35] (c_3);

\path [draw=black,line width= 1,->,>=latex] (c_2) edge[bend right=60] (c_1);

\path [draw=blue,line width= 1,->,>=latex,blue] (c_2) edge[bend right=90] (c_1);
\path [draw=blue,line width= 1,->,>=latex,blue] (c_1) edge[bend right=15] (c_3);
\path [draw=blue,line width= 1,->,>=latex,blue] (c_3) edge[bend right=50] (c_4);

\path [draw=red,line width= 1,->,>=latex,red] (c_4) edge[bend right=50] (c_3);
\path [draw=red,line width= 1,->,>=latex,red] (c_3) edge[bend right=35] (c_2);
\path [draw=red,line width= 1,->,>=latex,red] (c_2) edge[bend right=35] (c_1);

\end{tikzpicture}
\caption{The only double-move (apart from symmetry) that can be used to transform $\C$ into $\C'$.}
\end{subfigure}
\caption{A clustering-difference graph $\CDG$ for which the only way to transform $\C$ into $\C'$ in only two moves is to use the double-move from Lemma~\ref{lem:one_cycle_intersect}.}\label{fig:size_increase}
\end{figure}

When a path intersects multiple cycles from a set of disjoint cycles, integrating the corresponding moves as in the double-move from Lemma~\ref{lem:one_cycle_intersect} becomes more challenging since we can no longer guarantee that the second sequential move corresponds to a single directed path in the underlying $CDG$. However, we can ensure that transfers corresponding to at least the \textit{first} and \textit{last} edges of the path are applied in conjunction with the cycles. We show in the upcoming theorem that given any path $P=w_1...w_t$ and a set $\mathcal{Y}$ of disjoint cycles, a double-move can be used to correctly apply all transfers from $\mathcal{Y}$ while decreasing the cluster size corresponding to $w_1$ and increasing the cluster size corresponding to $w_t$. Furthermore, such a double-move will then allow us to completely integrate $P$ with $\mathcal{Y}$ as long as $P$ does not intersect the cycles of $\mathcal{Y}$ more than three times.

Recall that in a clustering-difference graph $\CDG$, the outdegree of a vertex is equal to the number of items which must be moved from the corresponding cluster to perform the clustering transformation. Note that in order for the outdegree to be reduced, a correct item must be \textit{sent} from the cluster to a new destination; however, this destination cluster need not actually be the other endpoint of the corresponding edge in $\CDG$. On the other hand, the indegree of a vertex gives the number of items which must be moved to the corresponding cluster. For the indegree to be reduced, a correct item must be \textit{received} by the vertex, but it does not matter which cluster actually sends the item. We call the minimum of the indegree and outdegree of a vertex the \textit{shared degree} of that vertex in the $CDG$. Hence, applying a set of disjoint cyclical moves reduces the shared degree of all covered vertices by one. When integrating a path with these cyclical moves, this reduction in shared degree should still occur in order to make all desired improvements to the underlying $CDG$. In the following theorem, we provide four different double-moves to accomplish this task. The type of double-move to use depends on the intersection points of the path with the cycles.

\begin{theorem}\label{thm:two_exchanges}
Let $\C, \C'$ be $k$-clusterings with clustering-difference graph $D := CDG(\C, \C')$, let $\mathcal{Y} = Y_1,...,Y_s$ be a set of disjoint cycles in $D$, and let $P = w_1...w_t$ be a path in $D$ that is edge-disjoint from $\mathcal{Y}$. There exists a double-move which accomplishes all of the following:
\begin{enumerate}
    \item Correctly applies all transfers from $\mathcal{Y}$
    \item Reduces the cluster size corresponding to $w_1$ through sending a correct item
    \item Increases the cluster size corresponding to $w_t$ through receiving a correct item
    \item Decreases the shared degree of each vertex covered by $\mathcal{Y}$ by at least one.
\end{enumerate}
\end{theorem}

\begin{proof}
If $s = 1$ or if $P$ is disjoint from $\mathcal{Y}$, we can apply Lemma~\ref{lem:one_cycle_intersect}, so assume $s \geq 2$. Let $w_{i_1}$ denote the first vertex of $P$ covered by $\mathcal{Y}$, and let $w_{i_2}, w_{i_3}$ denote, respectively, the second-to-last and last vertex of $P$ covered by $\mathcal{Y}$. Note that we have $w_{i_1} = w_{i_2}$ when $P$ intersects $\mathcal{Y}$ only twice. Similarly, if $P$ intersects $\mathcal{Y}$ only once, we let $w_{i_1} = w_{i_2} = w_{i_3}$. We treat four exhaustive cases regarding the distribution of $w_{i_1}$, $w_{i_2}$, and $w_{i_3}$ across the cycles of $\mathcal{Y}$. 

\vspace{.1in }
\textbf{Case 1:} $w_{i_1}$ and $w_{i_3}$ belong to different cycles of $\mathcal{Y}$ and $w_{i_2}$ belongs to the same cycle as $w_{i_3}$. We apply a cyclical move followed by a sequential move to perform all necessary transfers in $D$. See Figure~\ref{fig:case_1} for a visualization of the double-move.

Without loss of generality, assume $w_{i_1} \in Y_1$ and $w_{i_2}, w_{i_3} \in Y_s$. Choose an edge $e_i := (u_i, v_i)$ from each cycle $Y_i$. However, choose $e_1$ and $e_s$ so that $u_1 = w_{i_1}$ and $v_s = w_{i_2}$. For the cyclical move, first introduce an edge from $w_{i_1}$ to $w_{i_2} = v_s$ whose item corresponds to the item sent from $w_{i_1}$ in $P$. (Note that we might have $w_{i_1} = w_{i_2 - 1}$, in which case the edge already exists in $\CDG$.) Next, follow $Y_s$ from $v_s$ to $u_s$. Introduce an edge from $u_s$ to $v_{s-1}$ whose item corresponds to that of $e_s$, and then travel along $Y_{s-1} - e_{s-1}$ until $u_{s-1}$ is reached. Repeat for the remaining cycles by introducing edges $(u_i, v_{i-1})$ and traveling along $Y_{i-1} - e_{i-1}$ until $u_1 = w_{i_1}$ is reached to complete the cyclical move.

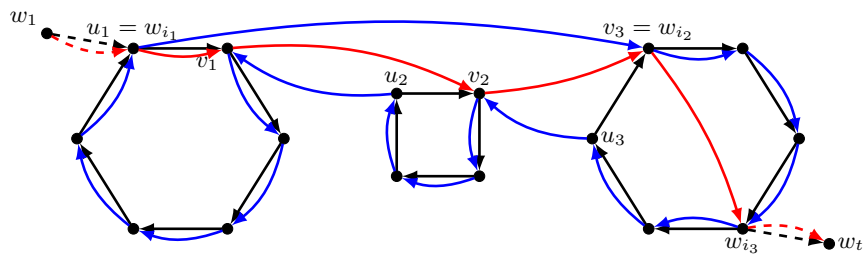
\begin{figure}
\centering
\begin{tikzpicture}[vertices/.style={draw, fill=black, circle, inner sep=0pt, minimum size = 4pt, outer sep=0pt}]

\node[vertices] (c_1) at (0,1.2) {};
\node[vertices] (c_2) at (0.75, 2.4) {};
\node[vertices] (c_3) at (2, 2.4) {};
\node[vertices] (c_4) at (2.75, 1.2) {};
\node[vertices] (c_5) at (2, 0) {};
\node[vertices] (c_6) at (0.75, 0) {};

\node[vertices] (e_1) at (4.25,0.7) {};
\node[vertices] (e_2) at (4.25, 1.8) {};
\node[vertices] (e_3) at (5.35, 1.8) {};
\node[vertices] (e_4) at (5.35, 0.7) {};

\node[vertices] (d_1) at (6.85,1.2) {};
\node[vertices] (d_2) at (7.6, 2.4) {};
\node[vertices] (d_3) at (8.85, 2.4) {};
\node[vertices] (d_4) at (9.6, 1.2) {};
\node[vertices] (d_5) at (8.85, 0) {};
\node[vertices] (d_6) at (7.6, 0) {};

\node[vertices] (w_1) at (-0.4,2.6) {};
\node[vertices] (w_2) at (10.0, -0.2) {};

\foreach \to/\from in {
	c_1/c_2, c_2/c_3, c_3/c_4, c_4/c_5, c_5/c_6, c_6/c_1,
	d_1/d_2, d_2/d_3, d_3/d_4, d_4/d_5, d_5/d_6, d_6/d_1,
	e_1/e_2, e_2/e_3, e_3/e_4, e_4/e_1}
\draw[draw=black, line width= 1,  ->, >=latex]  (\to)--(\from);

\path [draw=black,dashed,line width= 1,->,>=latex,black] (w_1) edge (c_2);
\path [draw=black,dashed,line width= 1,->,>=latex,black] (d_5) edge (w_2);

\node[above left] at (w_1) {$w_1$};
\node[right] at (w_2) {$w_t$};
\node[above] at (c_2) {$u_1 = w_{i_1}$};
\node[below left] at (c_3) {$v_1$};
\node[below] at (d_5) {$w_{i_3}$};
\node[right] at (d_1) {$u_3$};
\node[above] at (e_2) {$u_2$};
\node[above] at (e_3) {$v_2$};
\node[above] at (d_2) {$v_3 = w_{i_2}$};

\path [draw=blue,line width= 1,->,>=latex,blue] (c_2) edge[bend left=10] (d_2);
\path [draw=blue,line width= 1,->,>=latex,blue] (d_2) edge[bend right=20] (d_3);
\path [draw=blue,line width= 1,->,>=latex,blue] (d_3) edge[bend left=20] (d_4);
\path [draw=blue,line width= 1,->,>=latex,blue] (d_4) edge[bend left=20] (d_5);
\path [draw=blue,line width= 1,->,>=latex,blue] (d_5) edge[bend right=20] (d_6);
\path [draw=blue,line width= 1,->,>=latex,blue] (d_6) edge[bend left=20] (d_1);
\path [draw=blue,line width= 1,->,>=latex,blue] (d_1) edge[bend left=20] (e_3);
\path [draw=blue,line width= 1,->,>=latex,blue] (e_3) edge[bend right=20] (e_4);
\path [draw=blue,line width= 1,->,>=latex,blue] (e_4) edge[bend left=20] (e_1);
\path [draw=blue,line width= 1,->,>=latex,blue] (e_1) edge[bend left=20] (e_2);
\path [draw=blue,line width= 1,->,>=latex,blue] (e_2) edge[bend left=20] (c_3);
\path [draw=blue,line width= 1,->,>=latex,blue] (c_3) edge[bend right=20] (c_4);
\path [draw=blue,line width= 1,->,>=latex,blue] (c_4) edge[bend left=20] (c_5);
\path [draw=blue,line width= 1,->,>=latex,blue] (c_5) edge[bend left=20] (c_6);
\path [draw=blue,line width= 1,->,>=latex,blue] (c_6) edge[bend left=20] (c_1);
\path [draw=blue,line width= 1,->,>=latex,blue] (c_1) edge[bend right=20] (c_2);

\path [draw=red,dashed,line width= 1,->,>=latex,red] (w_1) edge[bend right=20] (c_2);
\path [draw=red,line width= 1,->,>=latex,red] (c_2) edge[bend right=15] (c_3);
\path [draw=red,line width= 1,->,>=latex,red] (c_3) edge[bend left=15] (e_3);
\path [draw=red,line width= 1,->,>=latex,red] (e_3) edge[bend right=10] (d_2);
\path [draw=red,line width= 1,->,>=latex,red] (d_2) edge[bend left=10] (d_5);
\path [draw=red,dashed,line width= 1,->,>=latex,red] (d_5) edge[bend left=20] (w_2);

\end{tikzpicture}
\caption{The double-move for Case 1 of Theorem~\ref{thm:two_exchanges}. The cycles of $\mathcal{Y}$ are given by the solid black edges, the first cyclical move is given by the blue edges, and the second sequential move is given by the red edges. The dashed edges indicate the sections of $P$ that exist when $w_1 \neq w_{i_1}$ or $w_{t} \neq w_{i_3}$. Note that when $w_{i_2} \neq w_{i_3 - 1}$, the second sequential move will include additional vertices of $P$ between $w_{i_2}$ and $w_{i_3}$ which are not covered by $\mathcal{Y}$.}\label{fig:case_1}
\end{figure}

Next, start a sequential move at $w_1$, following $P$ to $w_{i_1} = u_1$. Use the edge $(u_1, v_1)$. Then follow the edges $(v_i, v_{i+1})$ for $i = 1,...,s-1$ until $v_s = w_{i_2}$ is reached to correct items misplaced across the cycles. Complete the sequential move by following $P$ from $w_{i_2}$ to $w_{i_3}$ and then to $w_t$.

We now prove that the desired changes have been made to the underlying $CDG$. Clearly the cluster size corresponding to $w_1$ is decreased and the cluster size corresponding to $w_t$ is increased via the second sequential move. Furthermore, each item sent by $w_1$ or received by $w_t$ is correct. As in the double-move from Figure~\ref{fig:disjoint_cycles}, all edges from $\mathcal{Y}$ are applied through the combination of the two moves. Thus, it suffices to show that the shared degree of each vertex covered by $\mathcal{Y}$ has been reduced by at least one. The only interesting cases are $w_{i_1}$, $w_{i_2}$, and $w_{i_3}$. In the first cyclical move, $w_{i_1}$ both receives and sends a correct item, reducing its shared degree by one. In the following sequential move, either $w_{i_1}$ only sends a correct item (if $w_{i_1} = w_1$) or $w_{i_1}$ both sends and receives a correct item, so the net reduction in shared degree of $w_{i_1}$ is at least one. Similarly, $w_{i_3}$ sends and receives a correct item in the first move and then either receives a correct item or both receives and sends a correct item in the second move.

Finally, consider $w_{i_2}$. The vertex receives a possibly incorrect item from $w_{i_1}$ but also sends a correct item to its neighbor on $Y_s$ via the first cyclical move, leaving its shared degree, at worst, unchanged. In the second sequential move, $w_{i_2}$ receives a correct item originating from $u_s$ and then sends a correct item to the following vertex on $P$. Thus, the shared degree of $w_{i_2}$ is also reduced by at least one, as desired.

\vspace{.1in}

\textbf{Case 2:} $w_{i_1}$ and $w_{i_3}$ belong to different cycles of $\mathcal{Y}$ and $w_{i_2}$ does not belong to the same cycle as $w_{i_3}$. ($w_{i_2}$ may or may not belong to the same cycle as $w_{i_1}$, or we may have $w_{i_1} = w_{i_2}$.) We apply a sequential move followed by a cyclical move to perform the necessary transfers. See Figure~\ref{fig:case_2} for a visualization of the double-move.

Assume $w_{i_1} \in Y_1$ and $w_{i_3} \in Y_s$. Choose an edge $e_i := (u_i, v_i)$ from each $Y_i$ such that $v_1 = w_{i_1}$, $u_s = w_{i_3}$, and $v_i \neq w_{i_2}$ for $i = 2,...,s-1$.  First, travel along $P$ from $w_1$ to $w_{i_1}$. Then for $i = 1,...,s-1$, travel along $Y_i - e_i$ and follow an introduced edge $(u_i, v_{i+1})$. Finish the sequential move by following $Y_s - e_s$ from $v_s$ to $u_s = w_{i_3}$ and then following $P$ from $w_{i_3}$ to $w_t$.

\begin{figure}
\centering
\begin{tikzpicture}[vertices/.style={draw, fill=black, circle, inner sep=0pt, minimum size = 4pt, outer sep=0pt}]

\node[vertices] (c_1) at (0,1.2) {};
\node[vertices] (c_2) at (0.75, 2.4) {};
\node[vertices] (c_3) at (2, 2.4) {};
\node[vertices] (c_4) at (2.75, 1.2) {};
\node[vertices] (c_5) at (2, 0) {};
\node[vertices] (c_6) at (0.75, 0) {};

\node[vertices] (e_1) at (4.25,0.7) {};
\node[vertices] (e_2) at (4.25, 1.8) {};
\node[vertices] (e_3) at (5.35, 1.8) {};
\node[vertices] (e_4) at (5.35, 0.7) {};

\node[vertices] (d_1) at (6.85,1.2) {};
\node[vertices] (d_2) at (7.6, 2.4) {};
\node[vertices] (d_3) at (8.85, 2.4) {};
\node[vertices] (d_4) at (9.6, 1.2) {};
\node[vertices] (d_5) at (8.85, 0) {};
\node[vertices] (d_6) at (7.6, 0) {};

\node[vertices] (w_1) at (3.15,-0.2) {};
\node[vertices] (w_2) at (10.0, -0.2) {};

\foreach \to/\from in {
	c_1/c_2, c_2/c_3, c_3/c_4, c_4/c_5, c_5/c_6, c_6/c_1,
	d_1/d_2, d_2/d_3, d_3/d_4, d_4/d_5, d_5/d_6, d_6/d_1,
	e_1/e_2, e_2/e_3, e_3/e_4, e_4/e_1}
\draw[draw=black, line width= 1,  ->, >=latex]  (\to)--(\from);

\path [draw=black,dashed,line width= 1,->,>=latex,black] (w_1) edge[bend right=5] (c_5);
\path [draw=black,dashed,line width= 1,->,>=latex,black] (d_5) edge (w_2);

\node[below right] at (w_1) {$w_1$};
\node[below right] at (w_2) {$w_t$};
\node[below] at (c_5) {$w_{i_1}$};
\node[left] at (c_4) {$u_1$};
\node[below] at (d_5) {$w_{i_3}$};
\node[below] at (d_6) {$v_3$};
\node[below] at (e_4) {$u_2$};
\node[below] at (e_1) {$v_2$};
\node[above right] at (c_3) {$w_{i_2}$};

\path [draw=blue,dashed,line width= 1,->,>=latex,blue] (w_1) edge[bend right=25] (c_5);
\path [draw=blue,line width= 1,->,>=latex,blue] (c_5) edge[bend right=20] (c_6);
\path [draw=blue,line width= 1,->,>=latex,blue] (c_6) edge[bend left=20] (c_1);
\path [draw=blue,line width= 1,->,>=latex,blue] (c_1) edge[bend left=20] (c_2);
\path [draw=blue,line width= 1,->,>=latex,blue] (c_2) edge[bend left=20] (c_3);
\path [draw=blue,line width= 1,->,>=latex,blue] (c_3) edge[bend left=20] (c_4);
\path [draw=blue,line width= 1,->,>=latex,blue] (c_4) edge[bend left=10] (e_1);
\path [draw=blue,line width= 1,->,>=latex,blue] (e_1) edge[bend left=20] (e_2);
\path [draw=blue,line width= 1,->,>=latex,blue] (e_2) edge[bend left=20] (e_3);
\path [draw=blue,line width= 1,->,>=latex,blue] (e_3) edge[bend left=20] (e_4);
\path [draw=blue,line width= 1,->,>=latex,blue] (e_4) edge[bend left=10] (d_6);
\path [draw=blue,line width= 1,->,>=latex,blue] (d_6) edge[bend left=20] (d_1);;
\path [draw=blue,line width= 1,->,>=latex,blue] (d_1) edge[bend left=20] (d_2);
\path [draw=blue,line width= 1,->,>=latex,blue] (d_2) edge[bend left=20] (d_3);
\path [draw=blue,line width= 1,->,>=latex,blue] (d_3) edge[bend left=20] (d_4);
\path [draw=blue,line width= 1,->,>=latex,blue] (d_4) edge[bend left=20] (d_5);
\path [draw=blue,dashed,line width= 1,->,>=latex,blue] (d_5) edge[bend left=20] (w_2);

\path [draw=red,line width= 1,->,>=latex,red] (d_5) edge[bend right=20] (d_6);
\path [draw=red,line width= 1,->,>=latex,red] (d_6) edge[bend left=10] (e_1);
\path [draw=red,line width= 1,->,>=latex,red] (e_1) edge[bend right=10] (c_5);
\path [draw=red,line width= 1,->,>=latex,red] (c_5) edge[bend left=10] (c_3);
\path [draw=red,line width= 1,->,>=latex,red] (c_3) edge[bend left=25] (d_5);

\end{tikzpicture}
\caption{The double-move for Case 2 of Theorem~\ref{thm:two_exchanges}.}\label{fig:case_2}
\end{figure}
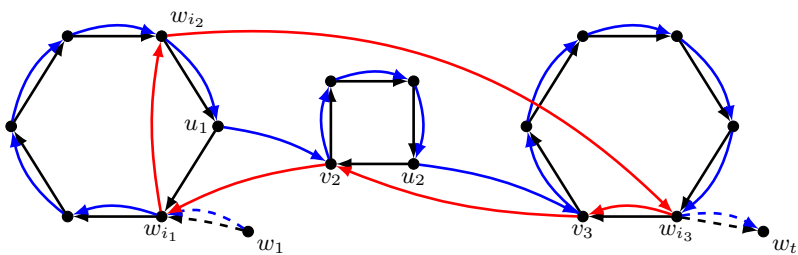

Start constructing the following cyclical move at $w_{i_3} = u_s$ using the edge $(u_s, v_s)$. Then for $i = s,...2$, follow the edges $(v_i, v_{i-1})$ until $v_1 = w_{i_1}$ is reached to correct items misplaced across the cycles. Now introduce (if needed) an edge from $w_{i_1}$ to $w_{i_2}$ whose item corresponds to the item sent from $w_{i_1}$ in $P$. Follow this edge to $w_{i_2}$, and then complete the cyclical move by following $P$ from $w_{i_2}$ to $w_{i_3}$. Note that this is indeed a single cyclical move since $w_{i_2} \neq v_i$ for $i = 2,...,s$.

The first sequential move alters the cluster sizes corresponding to $w_1$ and $w_t$ through correct transfers as desired, and again the edges of $\mathcal{Y}$ are all applied through the combination of the two moves. It suffices to show that the double-move reduces the shared degree of all vertices covered by $\mathcal{Y}$. However, this again follows from the argument used in the previous case: although $w_{i_2}$ may receive an incorrect item from $w_{i_1}$, it sends away two correct items and its shared degree is reduced by at least one.

\vspace{.1in}

\textbf{Case 3}: $w_{i_1}$ and $w_{i_3}$ belong to the same cycle of $\mathcal{Y}$ but $w_{i_2}$ belongs to a different cycle. Assume $w_{i_1}, w_{i_3} \in Y_1$ and $w_{i_2} \in Y_s$. We apply a cyclical move followed by a sequential move to perform the necessary transfers. See Figure~\ref{fig:case_3} for a visualization of the double-move.

Choose an edge $e_i := (u_i, v_i)$ from each cycle $Y_i$ such that $v_1 = w_{i_3}$ and $u_s = w_{i_2}$. For the cyclical move, first travel along $Y_1 - e_1$ from $v_1 = w_{i_3}$ to $u_1$. Then for $i = 1,...,s-1$, introduce and follow an edge $(u_i, v_{i+1})$ and follow $Y_{i+1} - e_{i+1}$ to $u_{i+1}$. Once $u_s = w_{i_2}$ is reached, follow $P$ from $w_{i_2}$ to $w_{i_3}$ to complete the cyclical move.

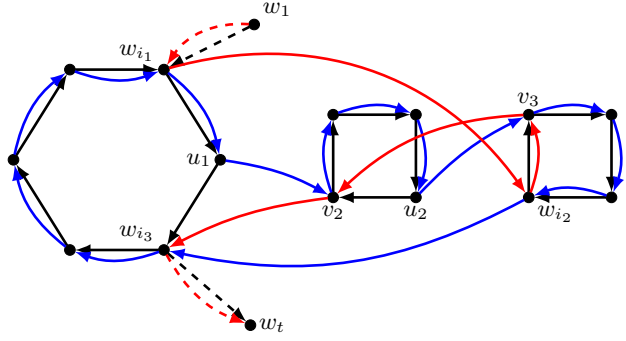
\begin{figure}
\centering
\begin{tikzpicture}[vertices/.style={draw, fill=black, circle, inner sep=0pt, minimum size = 4pt, outer sep=0pt}]

\node[vertices] (c_1) at (0,1.2) {};
\node[vertices] (c_2) at (0.75, 2.4) {};
\node[vertices] (c_3) at (2, 2.4) {};
\node[vertices] (c_4) at (2.75, 1.2) {};
\node[vertices] (c_5) at (2, 0) {};
\node[vertices] (c_6) at (0.75, 0) {};

\node[vertices] (e_1) at (4.25,0.7) {};
\node[vertices] (e_2) at (4.25, 1.8) {};
\node[vertices] (e_3) at (5.35, 1.8) {};
\node[vertices] (e_4) at (5.35, 0.7) {};

\node[vertices] (d_1) at (6.85,0.7) {};
\node[vertices] (d_2) at (6.85, 1.8) {};
\node[vertices] (d_3) at (7.95, 1.8) {};
\node[vertices] (d_4) at (7.95, 0.7) {};

\node[vertices] (w_1) at (3.2,3) {};
\node[vertices] (w_2) at (3.15,-1) {};

\foreach \to/\from in {
	c_1/c_2, c_2/c_3, c_3/c_4, c_4/c_5, c_5/c_6, c_6/c_1,
	d_1/d_2, d_2/d_3, d_3/d_4, d_4/d_1,
	e_1/e_2, e_2/e_3, e_3/e_4, e_4/e_1}
\draw[draw=black, line width= 1,  ->, >=latex]  (\to)--(\from);

\path [draw=black,dashed,line width= 1,->,>=latex,black] (w_1) edge (c_3);
\path [draw=black,dashed,line width= 1,->,>=latex,black] (c_5) edge (w_2);

\node[above right] at (w_1) {$w_1$};
\node[right] at (w_2) {$w_t$};
\node[above left] at (c_3) {$w_{i_1}$};
\node[left] at (c_4) {$u_1$};
\node[above left] at (c_5) {$w_{i_3}$};;
\node[below] at (e_4) {$u_2$};
\node[below] at (e_1) {$v_2$};
\node[below right] at (d_1) {$w_{i_2}$};
\node[above] at (d_2) {$v_3$};

\path [draw=blue,line width= 1,->,>=latex,blue] (c_5) edge[bend left=20] (c_6);
\path [draw=blue,line width= 1,->,>=latex,blue] (c_6) edge[bend left=20] (c_1);
\path [draw=blue,line width= 1,->,>=latex,blue] (c_1) edge[bend left=20] (c_2);
\path [draw=blue,line width= 1,->,>=latex,blue] (c_2) edge[bend right=20] (c_3);
\path [draw=blue,line width= 1,->,>=latex,blue] (c_3) edge[bend left=20] (c_4);
\path [draw=blue,line width= 1,->,>=latex,blue] (c_4) edge[bend left=10] (e_1);
\path [draw=blue,line width= 1,->,>=latex,blue] (e_1) edge[bend left=20] (e_2);
\path [draw=blue,line width= 1,->,>=latex,blue] (e_2) edge[bend left=20] (e_3);
\path [draw=blue,line width= 1,->,>=latex,blue] (e_3) edge[bend left=20] (e_4);
\path [draw=blue,line width= 1,->,>=latex,blue] (e_4) edge[bend left=10] (d_2);
\path [draw=blue,line width= 1,->,>=latex,blue] (d_2) edge[bend left=20] (d_3);
\path [draw=blue,line width= 1,->,>=latex,blue] (d_3) edge[bend left=20] (d_4);
\path [draw=blue,line width= 1,->,>=latex,blue] (d_4) edge[bend right=20] (d_1);
\path [draw=blue,line width= 1,->,>=latex,blue] (d_1) edge[bend left=20] (c_5);

\path [draw=red,dashed,line width= 1,->,>=latex,red] (w_1) edge[bend right=30] (c_3);
\path [draw=red,line width= 1,->,>=latex,red] (c_3) edge[bend left=35] (d_1);
\path [draw=red,line width= 1,->,>=latex,red] (d_1) edge[bend right=20] (d_2);
\path [draw=red,line width= 1,->,>=latex,red] (d_2) edge[bend right=20] (e_1);
\path [draw=red,line width= 1,->,>=latex,red] (e_1) edge[bend right=10] (c_5);
\path [draw=red,dashed,line width= 1,->,>=latex,red] (c_5) edge[bend right=20] (w_2);

\end{tikzpicture}
\caption{The double-move for Case 3 of Theorem~\ref{thm:two_exchanges}.}\label{fig:case_3}
\end{figure}

Start the sequential move by following $P$ from $w_1$ to $w_{i_1}$. Next, introduce (if needed) an edge from $w_{i_1}$ to $w_{i_2} = u_s$ whose item corresponds to the item sent from $w_{i_1}$ in $P$. Follow this edge and then follow the edge $(u_s, v_s)$. Next, for $i = s,...,2$, follow the edge $(v_i, v_{i-1})$ to correct items misplaced among the cycles. Once $v_1 = w_{i_3}$ is reached, finish the sequential move by following $P$ from $w_{i_3}$ to $w_t$. All clusters are then changed as desired using the arguments from Case 1.

\vspace{.1in}

\textbf{Case 4}: $w_{i_1}$, $w_{i_2}$, and $ w_{i_3}$ belong to the same cycle in $\mathcal{Y}$ (allowing for either $w_{i_1} = w_{i_2}$ or $w_{i_1} = w_{i_2} = w_{i_3}$).
We apply two sequential moves similar to those of Lemma~\ref{lem:one_cycle_intersect} to perform all necessary transfers. See Figure~\ref{fig:case_4} for a visualization of the double-move. 

Assume $w_{i_1}, w_{i_2}, w_{i_3}\in Y_1$.  Choose an edge $e_i := (u_i, v_i)$ from each cycle $Y_i$, choosing $e_1$ so that $v_1 = w_{i_1}$. For the first sequential move, first follow $P$ from $w_1$ to $w_{i_1} = v_1$ and then travel along $Y_1 - e_1$ to $u_1$. Then for $i = 1,...,s-1$, introduce the edge $(u_i, v_{i+1})$ and travel along $Y_i - e_i$. Terminate once $u_s$ is reached.

\begin{figure}
\centering
\begin{tikzpicture}[vertices/.style={draw, fill=black, circle, inner sep=0pt, minimum size = 4pt, outer sep=0pt}]

\node[vertices] (c_1) at (0,1.2) {};
\node[vertices] (c_2) at (0.75, 2.4) {};
\node[vertices] (c_3) at (2, 2.4) {};
\node[vertices] (c_4) at (2.75, 1.2) {};
\node[vertices] (c_5) at (2, 0) {};
\node[vertices] (c_6) at (0.75, 0) {};

\node[vertices] (w_1) at (3.3,-0.2) {};
\node[vertices] (w_2) at (2.75, 1.2) {};
\node[vertices] (w_4) at (0, 1.2) {};
\node[vertices] (w_5) at (-0.6, -0.2) {};

\node[vertices] (d_1) at (4.1,0.8) {};
\node[vertices] (d_2) at (4.8, 1.9) {};
\node[vertices] (d_3) at (5.5, 0.8) {};

\node[vertices] (e_1) at (6.7,0.7) {};
\node[vertices] (e_2) at (6.7, 1.8) {};
\node[vertices] (e_3) at (7.8, 1.8) {};
\node[vertices] (e_4) at (7.8, 0.7) {};

\foreach \to/\from in {
	c_1/c_2, c_2/c_3, c_3/c_4, c_4/c_5, c_5/c_6, c_6/c_1,
	d_1/d_2, d_2/d_3, d_3/d_1,
	e_1/e_2, e_2/e_3, e_3/e_4, e_4/e_1}
\draw[draw=black, line width= 1,  ->, >=latex]  (\to)--(\from);

\node[right] at (w_1) {$w_1$};
\node[above right] at (w_2) {$w_{i_1}$};
\node[left] at (w_5) {$w_t$};
\node[above] at (c_3) {$u_1$};
\node[below left] at (d_1) {$v_2$};
\node[below] at (d_3) {$u_2$};
\node[below] at (e_1) {$v_3$};
\node[below right] at (e_4) {$u_3$};
\node[below] at (c_6) {$w_{i_3}$};
\node[above left] at (c_2) {$w_{i_2}$};

\path [draw=black,dashed,line width= 1,->,>=latex,black] (w_1) edge (w_2);
\path [draw=black,dashed,line width= 1,->,>=latex,black] (c_6) edge (w_5);

\path [draw=blue,dashed,line width= 1,->,>=latex,blue] (w_1) edge[bend right=20] (w_2);
\path [draw=blue,line width= 1,->,>=latex,blue] (c_4) edge[bend left=20] (c_5);
\path [draw=blue,line width= 1,->,>=latex,blue] (c_5) edge[bend left=20] (c_6);
\path [draw=blue,line width= 1,->,>=latex,blue] (c_6) edge[bend left=20] (c_1);
\path [draw=blue,line width= 1,->,>=latex,blue] (c_1) edge[bend left=20] (c_2);
\path [draw=blue,line width= 1,->,>=latex,blue] (c_2) edge[bend left=20] (c_3);
\path [draw=blue,line width= 1,->,>=latex,blue] (c_3) edge[bend left=20] (d_1);
\path [draw=blue,line width= 1,->,>=latex,blue] (d_1) edge[bend left=20] (d_2);
\path [draw=blue,line width= 1,->,>=latex,blue] (d_2) edge[bend left=20] (d_3);
\path [draw=blue,line width= 1,->,>=latex,blue] (d_3) edge[bend left=20] (e_1);
\path [draw=blue,line width= 1,->,>=latex,blue] (e_1) edge[bend left=20] (e_2);
\path [draw=blue,line width= 1,->,>=latex,blue] (e_2) edge[bend left=20] (e_3);
\path [draw=blue,line width= 1,->,>=latex,blue] (e_3) edge[bend left=20] (e_4);

\path [draw=red,line width= 1,->,>=latex,red] (e_4) edge[bend left=20] (e_1);
\path [draw=red,line width= 1,->,>=latex,red] (e_1) edge[bend left=20] (d_1);
\path [draw=red,line width= 1,->,>=latex,red] (d_1) edge[bend right=10] (w_2);
\path [draw=red,line width= 1,->,>=latex,red] (w_2) edge[bend left=10] (c_2);
\path [draw=red,line width= 1,->,>=latex,red] (c_2) edge[bend left=20] (c_6);
\path [draw=red,dashed,line width= 1,->,>=latex,red] (c_6) edge[bend right=20] (w_5);

\end{tikzpicture}
\caption{The double-move for Case 4 of Theorem~\ref{thm:two_exchanges}.}\label{fig:case_4}
\end{figure}
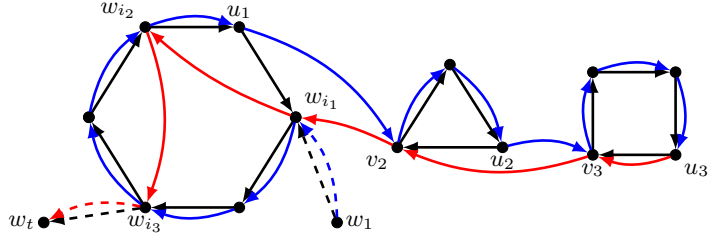

For the second sequential move, begin at $u_s$ via edge $e_s$. Then for $i = s,...,2$, travel along the edge $(v_i, v_{i-1})$ to correct items misplaced among the cycles. Once $v_1 = w_{i_1}$ is reached, introduce (if needed) an edge from $w_{i_1}$ to $w_{i_2}$ whose item corresponds to the item sent from $w_{i_1}$ in $P$. Follow this edge and then finish the move by following $P$ from $w_{i_2}$ to $w_{i_3}$ and then to $w_t$. All clusters are then changed as desired by again using the arguments from Case 1. \eoproof
\end{proof}

In the proof of Theorem~\ref{thm:two_exchanges}, we observe an important implication: if a path intersects a set of disjoint cycles at most three times, then \textit{all} transfers correcponding to the cycles and the path can be correctly applied using one of the double-moves from the theorem.

\begin{theorem}\label{thm:three_intersections}
Let $\C, \C'$ be $k$-clusterings where $\CDG$ consists of a set $\mathcal{Y}$ of disjoint cycles and a path $P$. If $P$ intersects $\mathcal{Y}$ at most three times, then $d(\C, \C') = 2$. 
\end{theorem}
\begin{proof}
Let $P = w_1...w_t$. As in the proof of Theorem~\ref{thm:two_exchanges}, consider $w_{i_1}$, $w_{i_2}$, and $w_{i_3}$: the first, second-to-last, and last vertices of $P$ which are covered by $\mathcal{Y}$. If these are the only vertices of $P$ covered by $\mathcal{Y}$, then all edges of $P$ can be applied in any of the four double-moves from Theorem~\ref{thm:two_exchanges}. To see this, note that since no vertices on $P$ between $w_{i_1}$ and $w_{i_2}$ are covered by $\mathcal{Y}$, we can follow \textit{all} corresponding edges of $P$ when sending an item from $w_{i_1}$ to $w_{i_2}$. As in the proof of the theorem, the same holds for all edges between $w_{i_2}$ and $w_{i_3}$. Hence, all transfers corresponding to the edges from both $\mathcal{Y}$ and $P$ are correctly applied through the appropriate double-move.
\eoproof
\end{proof}

Of course, since the number of times a path $P$ intersects a set $\mathcal{Y}$ of cycles is at most the number of vertices in $P$, this implies that a path with at most three vertices can always be integrated with $\mathcal{Y}$ using one of the double-moves from Theorem~\ref{thm:two_exchanges}.

\begin{corollary}\label{cor:three_vertices}
Let $\C, \C'$ be $k$-clusterings where $\CDG$ consists of a set $\mathcal{Y}$ of disjoint cycles and a path $P$ with at most three vertices. Then $d(\C, \C') = 2$. 
\end{corollary}

\section{Bounds on the Transformation Distance}\label{sec:bounds}

In this section, we use the double-moves from Section~\ref{sec:moves} to prove upper bounds on the transformation distance between clusterings based on certain properties of the related $CDG$. 

Given any two $k$-clusterings $\C, \C'$ with clustering-difference graph $D := \CDG$, our goal is to transform $\C$ into $\C'$ using as few cyclical and sequential moves as possible. Recall the fundamental difference between these two types of moves: cyclical moves transfer items among the clusters while preserving the original cluster sizes; on the other hand, sequential moves transfer items while increasing the size of one cluster and decreasing the size of another. This motivates a decomposition of $D$ into two parts corresponding to these different types of moves.

\begin{definition}[Path-Cycle Decomposition]
Let $\C, \C'$ be $k$-clusterings with clustering-difference graph $D := \CDG$. For $i = 1,...,k$, let $\delta_i$ denote $\left| (|C_i| - |C'_i|) \right|$, the change in the size of cluster $C_i$ between $\C$ and $\C'$. A \textit{path-cycle decomposition $(D_P, D_Y)$ of $D$} is a decomposition of $D$ into two parts: a set $D_P$ containing $ \frac{1}{2}\sum_{i=1}^k \delta_i$ directed paths and a graph $D_Y$ which decomposes into directed cycles.
\end{definition}

For any path-cycle decomposition $(D_P, D_Y)$ of $D$, the paths of $D_P$ adjust the cluster sizes of $\C$ to those of $\C'$ and the edges of $D_Y$ apply any remaining transfers. Such a decomposition can be found easily: greedily construct directed paths in $D$ which begin at \textit{excess vertices}, those vertices $c_i$ satisfying $d^+(c_i) > d^-(c_i)$, and terminate at \textit{deficit vertices}, those satisfying $d^-(c_i) > d^+(c_i)$, and add the paths to $D_P$. Once there do not exist any excess or deficit vertices in $D$, the remaining edges in the graph form $D_Y$. Alternatively, greedily remove directed cycles from $D$ to build $D_Y$ and the leftover edges will decompose into $D_P$. Note that we can store $D_P$ either as a set of directed paths or as a graph which decomposes into paths. Nevertheless, the fixed number of paths from excess to deficit vertices in $D_P$ gives a lower bound on the transformation distance between the clusterings.

\begin{lemma}\label{lem:lower_bound}
Let $\C, \C'$ be $k$-clusterings of the same data set. Then $d(\C, \C') \geq \frac{1}{2}\sum_{i=1}^k \delta_i$, where $\delta_i = | (|C_i| - |C'_i|) |$.
\end{lemma}
\begin{proof}
By definition, $\delta_i$ is the change in cluster size needed to transform $C_i$ into $C'_i$. Note that the sum $\sum_{i=1}^k \delta_i$ is therefore even. Cyclical moves do not change the size of any clusters while sequential moves change the size of exactly two clusters by one. Hence, at least $\frac{1}{2}\sum_{i=1}^k \delta_i$ sequential moves are required in order to change the cluster sizes of $\C$ to those of $\C'$.  \eoproof
\end{proof}

Given a path-cycle decomposition $(D_P, D_Y)$ of $D$, a straightforward approach for transforming $\C$ into $\C'$ is to separately apply the paths of $D_P$ followed by the cycles of $D_Y$. However, whereas a fixed number of sequential moves is required to apply all paths of $D_P$, the number of cyclical moves required to apply all transfers in $D_Y$ is generally less than its number of cycles. Using the double-move from Figure~\ref{fig:disjoint_cycles}, we can integrate sets of disjoint cycles from $D_Y$ to achieve a transformation distance bound which generalizes Corollary 7 in \cite{b-13}. This serves as a starting point for our discussion on an improved upper bound for $d(\C, \C')$. Recall that the \textit{shared degree} $\eta_i$ of a vertex $c_i$ in $D$ is the minimum of its indegree and outdegree.

\begin{lemma}\label{lem:naive_distance_bound}
Let $\C, \C'$ be $k$-clusterings of the same data set. Then
\begin{align*}
d(\C, \C') \leq \eta_{i_1} + \eta_{i_2} + \frac{1}{2}\sum_{i=1}^k \delta_i,
\end{align*}
where $\delta_i = | (|C_i| - |C'_i|) |$, $\eta_i$ is the shared degree of $c_i$ in $\CDG$, $i_1 = \arg \max \eta_i$, and $i_2 = \arg \max_{i \neq i_1} \eta_i$. 
\end{lemma}
\begin{proof}
Let $(D_P, D_Y)$ be any path-cycle decomposition of $D := \CDG$. Applying the $\frac{1}{2}\sum_{i=1}^k \delta_i$ sequential moves corresponding to the paths of $D_P$ correctly adjusts all cluster sizes. 

Next, we can use the method of Corollary~7 from \cite{b-13} to apply the cycles in $D_Y$. To do so, note that for $i=1,...,k$, the shared degree of $c_i$ in $D_Y$ is at most $\eta_i$. Hence, we may first apply at most $\eta_{i_1} - \eta_{i_2}$ cyclical moves to reduce the maximum shared degree in $D_Y$ to at most $\eta_{i_2}$. Next, using the technique in Corollary~3 from \cite{b-13}, we can solve a maximum flow problem to obtain a set of disjoint cycles in $D_Y$ covering all vertices of maximum shared degree. All transfers from this cycle cover can be applied via at most two cyclical moves using the double-move from Figure~\ref{fig:disjoint_cycles}. Repeating until the maximum shared degree of $D_Y$ is zero, all transfers from $D_Y$ are performed in at most $(\eta_{i_1} - \eta_{i_2}) + 2 \eta_{i_2} = \eta_{i_1} + \eta_{i_2}$ cyclical moves. Therefore, at most $\eta_{i_1} + \eta_{i_2} + \frac{1}{2}\sum_{i=1}^k \delta_i$ cyclical and sequential moves are required to transform $\C$ into $\C'$.
 \eoproof
\end{proof}

This initial upper bound on $d(\C, \C')$ uses the double-move for integrating disjoint cyclical moves from Figure~\ref{fig:disjoint_cycles} but does not yet take advantage of any of the double-moves from Section~\ref{sec:moves} which integrate both cyclical and sequential moves. For instance, when applying the cyclical moves from a disjoint cycle cover $\mathcal{Y}$ of all vertices of maximum shared degree in $D_Y$, we could attempt to integrate a path $P$ from $D_P$. If $P$ is disjoint from $\mathcal{Y}$, if $P$ intersects at most one cycle of $\mathcal{Y}$, or if $P$ intersects $\mathcal{Y}$ at most three times, we could use one of the double-moves from Lemma~\ref{lem:disjoint_paths}, Lemma~\ref{lem:one_cycle_intersect}, or Theorem~\ref{thm:three_intersections} to integrate $P$ at no extra cost, reducing the number of remaining sequential moves. However, we cannot guarantee that such a path $P$ exists in $D_P$. 

Nevertheless, we can achieve an improved bound on the transformation distance by considering each path in $D_P$ as the combination of a (short) sequential move with a cyclical move. To motivate this, suppose $\CDG$ consists of a set $\mathcal{Y}$ of disjoint cycles and a path $P = w_1...w_t$ with $t \geq 4$, where $w_1$, $w_{t-1}$, and $w_t$ are all covered by $\mathcal{Y}$. We can apply one of the four double-moves from Theorem~\ref{thm:two_exchanges} to perform all transfers corresponding to the edges in $\mathcal{Y}$ while decreasing the excess of $w_1$ and the deficit of $w_t$. However, the double-move does not apply any edges of $P$ between $w_2$ and $w_{t-1}$. Additionally, $w_{t-1}$ receives an incorrect item from $w_1$ during the double-move which still needs to be sent to $w_2$. Hence, a new edge from $w_{t-1}$ to $w_{2}$ is introduced and the resulting $CDG$ consists of the directed cycle $w_2...w_{t-1}w_2$. 

In this manner, we can represent the sequential move corresponding to any path $P = w_1...w_t$ in $D_P$ covering $t \geq 4$ vertices as the combination of such a cycle $P_Y = w_2...w_{t-1}w_2$ and a path $P'$ with three vertices. As depicted in Figure~\ref{fig:path_cycle}, there are two cases to consider regarding the interior vertex of $P'$ and the item sent along the artificially introduced edge $e_P := (w_{t-1}, w_2)$ in $P_Y$. These cases depend on the order in which the corresponding transfers are applied. Let $x_2$ denote the item to be sent from $w_1$ to $w_2$ in $P$, and let $x_t$ denote the item to be sent from $w_{t-1}$ to $w_t$. If $P'$ is applied first (as in the previous paragraph), let $P' := w_1w_{t-1}w_t$ and let $e_P$ send $x_2$ from $w_{t-1}$ to the correct destination $w_2$. This case is depicted in Figure~\ref{fig:path_cycle_1}. On the other hand, if $e_P$ is applied before $P'$, let $e_P$ send $x_t$ from $w_{t-1}$ to $w_2$ and let $P' := w_1w_2w_t$ as depicted in Figure~\ref{fig:path_cycle_2}. Either case has the same effect on the underlying clusters as the original path $P$.

\begin{figure}[b]
\centering
\begin{subfigure}{0.45 \textwidth}
\centering
\begin{tikzpicture}[vertices/.style={draw, fill=black, circle, inner sep=0pt, minimum size = 4pt, outer sep=0pt}, scale=1.2]

\node[vertices] (w_1) at (0,0) {};
\node[vertices] (w_2) at (1,0) {};
\node[vertices] (w_3) at (2,0) {};
\node[vertices] (w_4) at (3,0) {};
\node[vertices] (w_5) at (4,0) {};
\node[vertices] (w_6) at (5,0) {};

\foreach \to/\from in {w_1/w_2, w_2/w_3, w_4/w_5, w_5/w_6}
\draw[draw=black, line width= 1,  ->, >=latex]  (\to)--(\from);
\path [draw=black, dashed, line width= 1,->,>=latex] (w_3) edge (w_4);

\node[below] at (w_1) {$w_1$};
\node[below] at (w_2) {$w_2$};
\node[below] at (w_5) {$w_{t-1}$};
\node[below] at (w_6) {$w_t$};

\node (x_2_2) at (0.7, 0.35) {$x_2$};
\node (x_t_2) at (4.5, 0.26) {$x_t$};
\node (x_both) at (2.5, 0.95) {$x_2$};

\path [draw=blue,line width= 1,->,>=latex,blue] (w_1) edge[bend left=20] (w_5);
\path [draw=blue,line width= 1,->,>=latex,blue] (w_5) edge[bend left=25] (w_6);
\path [draw=green,line width= 1,->,>=latex,green] (w_5) edge[bend right=65] (w_2);
\path [draw=green,line width= 1,->,>=latex,green] (w_2) edge[bend left=15] (w_3);
\path [draw=green, dashed, line width= 1,->,>=latex,green] (w_3) edge[bend left=15] (w_4);
\path [draw=green,line width= 1,->,>=latex,green] (w_4) edge[bend left=15] (w_5);
\end{tikzpicture}
\caption{The first case for $P'$ and $e_P$.}\label{fig:path_cycle_1}
\end{subfigure}
\quad
\begin{subfigure}{0.45 \textwidth}
\centering
\begin{tikzpicture}[vertices/.style={draw, fill=black, circle, inner sep=0pt, minimum size = 4pt, outer sep=0pt}, scale=1.2]

\node[vertices] (w_1) at (0,0) {};
\node[vertices] (w_2) at (1,0) {};
\node[vertices] (w_3) at (2,0) {};
\node[vertices] (w_4) at (3,0) {};
\node[vertices] (w_5) at (4,0) {};
\node[vertices] (w_6) at (5,0) {};

\foreach \to/\from in {w_1/w_2, w_2/w_3, w_4/w_5, w_5/w_6}
\draw[draw=black, line width= 1,  ->, >=latex]  (\to)--(\from);
\path [draw=black, dashed, line width= 1,->,>=latex] (w_3) edge (w_4);

\node[below] at (w_1) {$w_1$};
\node[below] at (w_2) {$w_2$};
\node[below] at (w_5) {$w_{t-1}$};
\node[below] at (w_6) {$w_t$};

\node (x_2_1) at (0.5, 0.25) {$x_2$};
\node (x_t_1) at (2.6, 0.485) {$x_t$};
\node (x_both) at (2.5, 0.95) {$x_t$};

\path [draw=red,line width= 1,->,>=latex,red] (w_1) edge[bend left=20] (w_2);
\path [draw=red,line width= 1,->,>=latex,red] (w_2) edge[bend left=20] (w_6);
\path [draw=green,line width= 1,->,>=latex,green] (w_5) edge[bend right=65] (w_2);
\path [draw=green,line width= 1,->,>=latex,green] (w_2) edge[bend left=15] (w_3);
\path [draw=green, dashed, line width= 1,->,>=latex,green] (w_3) edge[bend left=15] (w_4);
\path [draw=green,line width= 1,->,>=latex,green] (w_4) edge[bend left=15] (w_5);
\end{tikzpicture}
\caption{The second case for $P'$ and $e_P$.}\label{fig:path_cycle_2}
\end{subfigure}
\caption{The sequential move corresponding to the path $P=w_1...w_t$, given by the black edges, is equivalent to one of two combinations of a short sequential move with a cyclical move. In the first case, the path $P' := w_1w_{t-1}w_t$, given by the blue edges, is applied first. Next, the cycle $P_Y$, given by the green edges, is applied and sends $x_2$ to the correct destination $w_2$ via the introduced edge $e_P := (w_{t-1}, w_2)$. In the second case, the cycle $P_Y$ is applied first, sending $x_t$ from $w_{t-1}$ to $w_2$ via $e_P$. Then the path $P' := w_1 w_2 w_t$, given by the red edges, is applied and $x_t$ is correctly sent from $w_2$ to $w_t$. In both cases, all transfers corresponding to the original path $P$ are performed correctly via the path $P'$ and the cycle $P_Y$.}\label{fig:path_cycle}
\end{figure}
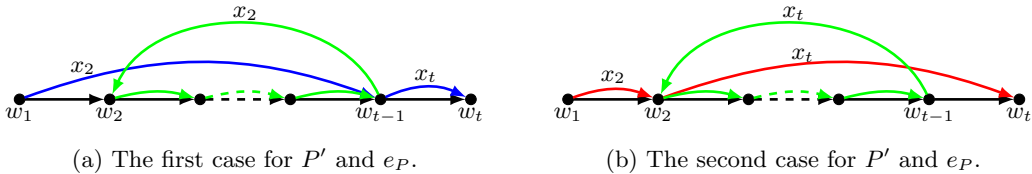

Therefore, we decompose each path $P$ from $D_P$ with more than three vertices in this manner, adding the resulting cycle $P_Y$ to $D_Y$ and replacing $P$ with $P'$ in $D_P$. All paths in $D_P$ then have at most three vertices and Corollary~\ref{cor:three_vertices} implies that we can completely integrate any of these paths with a disjoint cycle cover from $D_Y$ in only two moves. Note also that each vertex in a cycle $P_Y$ is an interior vertex of the original path $P$. Hence, even after introducing these additional cycles to $D_Y$, the shared degree of each vertex in $D_Y$ remains at most the shared degree of that vertex in the original clustering-difference graph. This allows us to improve upon the distance bound given in Lemma~\ref{lem:naive_distance_bound}.

The challenge in this approach lies in the fact that the interior vertex of $P'$ (either $w_{t-1}$ or $w_{2}$) and the label of $e_P$ (either $x_2$ or $x_t$) depend on the order in which the corresponding transfers are applied. If a cycle cover $\mathcal{Y}$ does not include $e_P$, then integrating $P'$ with $\mathcal{Y}$ is straightforward -- simply choose the correct interior vertex for $P'$ depending on whether or not $e_P$ has been applied yet. Once $P'$ is applied, then if $e_P$ remains in $D_Y$, change its label from $x_t$ to $x_2$. 

However, if $e_P$ is contained in $\mathcal{Y}$, then we must make adjustments to the double-moves from Theorem~\ref{thm:two_exchanges} in order to take into account the different possible cases for $e_P$ and $P'$. Nevertheless, this approach allows us to integrate sequential moves from $D_P$ with disjoint cyclical moves from $D_Y$ at no extra cost, resulting in the following greatly improved distance bound. The bound depends only on the larger of the second-largest shared degree and the overall change in cluster sizes rather than on the sum of these values as in Lemma~\ref{lem:naive_distance_bound}.

\begin{theorem}\label{thm:improved_bound}
Let $\C, \C'$ be $k$-clusterings of the same data set. Then
\begin{align*}
d(\C, \C') \leq \eta_{i_1} + \max \left\lbrace \eta_{i_2}, \ \frac{1}{2}\sum_{i=1}^k \delta_i \right\rbrace,
\end{align*}
where $\delta_i = | (|C_i| - |C'_i|) |$, $\eta_i$ is the shared degree of $c_i$ in $\CDG$, $i_1 = \arg \max \eta_i$, and $i_2 = \arg \max_{i \neq i_1} \eta_i$.
\end{theorem}
\begin{proof}
Let $(D_P, D_Y)$ be any path-cycle decomposition of $D := \CDG$. For each path $P = w_1...w_t$ of the $\frac{1}{2}\sum_{i=1}^k \delta_i$ paths in $D_P$, if $t \geq 4$, decompose $P$ into a cycle $P_Y$ and a short path $P'$ as depicted in the cases of Figure~\ref{fig:path_cycle}. Specifically, let $P' := w_1w_{t-1}w_t$  where the label of edge $(w_1, w_{t-1})$ is the item $x_2$ to be sent from $w_1$ to $w_2$ in $P$, as depicted by the blue edges in Figure~\ref{fig:path_cycle_1}. Replace $P$ with $P'$ in $D_P$. In addition, introduce the cycle $P_Y = w_2...w_{t-1}w_2$ to $D_Y$, where the label of the artificial edge $e_P := (w_{t-1}, w_2)$ in $P_Y$ is the item $x_t$ to be sent from $w_{t-1}$ to $w_t$ in $P$, as depicted by the green edges in Figure~\ref{fig:path_cycle_2}. Note that each vertex in $P_Y$ is an interior vertex of the original path $P$; hence, in the resulting cycle graph $D_Y$, the shared degree of each vertex $c_i$ remains at most $\eta_i$. 

As in the proof of Lemma~\ref{lem:naive_distance_bound}, first apply at most $\eta_{i_1} - \eta_{i_2}$ cyclical moves to reduce the maximum shared degree in $D_Y$ to at most $\eta_{i_2}$. Whenever an artificial edge $e_P$ is applied in such a move, change the interior vertex of the corresponding path $P'$ in $D_P$ so that $P' = w_1w_2w_{t}$ as in Figure~\ref{fig:path_cycle_2}. 

Now, again as in the proof of Lemma~\ref{lem:naive_distance_bound}, we can reduce the maximum shared degree in $D_Y$ by finding a disjoint cycle cover for the vertices of maximum shared degree and applying a double-move. However, in each such double-move, we will also integrate a path from $D_P$. 

Let $\mathcal{Y}$ be such a set of disjoint cycles in $D_Y$, which can be found using the technique from \cite{b-13}. Choose any path from $D_P$. Since each path in $D_P$ has at most three vertices, if the selected path is an original path from $\CDG$, integrating the path with $\mathcal{Y}$ in a double-move is straightforward via  Corollary~\ref{cor:three_vertices}. Again, if an artificial edge $e_P$ is applied through this double-move and the corresponding path $P'$ remains in $D_P$, switch the interior vertex of $P'$ from $w_{t-1}$ to $w_2$ as in Figure~\ref{fig:path_cycle_2}.

Hence, assume the selected path from $D_P$ is an introduced path of the form $P'$ with corresponding artificial edge $e_P$. If $e_P$ is not contained in the cycle cover $\mathcal{Y}$, integrating $P'$ with $\mathcal{Y}$ is again straightforward via Corollary~\ref{cor:three_vertices}: the interior vertex of $P'$ is known and we can simply apply one of the double-moves from Theorem~\ref{thm:two_exchanges}. After the double-move is applied, make any necessary adjustments to the remaining paths in $D_P$ as in the previous paragraph. Additionally, if $e_P$ remains in $D_Y$, change its label from $x_t$ to $x_2$ as in Figure~\ref{fig:path_cycle_1}.

However, if the edge $e_P$ corresponding to $P'$ is contained in $\mathcal{Y}$, then we must make modifications to the double-moves of Theorem~\ref{thm:two_exchanges} to account for the two different cases for $e_P$ and $P'$. Note that if this situation arises, $e_P$ has not yet been applied so $P'$ has the initial form $P' = w_1w_{t-1}w_t$. We modify each of the four cases regarding the intersection points of $P'$ with $\mathcal{Y}$ from Theorem~\ref{thm:two_exchanges} to perform the necessary transfers. Since $e_P$ is included in $\mathcal{Y}$, both $w_{t-1}$ and $w_2$ are necessarily covered by $\mathcal{Y}$, but $w_1$ and $w_t$ need not be covered. Several of the case modifications depend on whether or not these two vertices are actually covered by the cycles.

\vspace{.1in}
\textbf{Case 1:} All three vertices of $P' = w_1 w_{t-1} w_t$ are covered by $\mathcal{Y}$, where $w_1$ and $w_t$ belong to different cycles of $\mathcal{Y}$ and $w_{t-1}$ (and hence, also $e_P$ and $w_2$) belongs to the same cycle as $w_t$. See the examples in Figure~\ref{fig:case_1_modified} -- the artificial edge $e_P$ is given by the green edge from $w_{t-1}$ to $w_2$ and the other edges of $\mathcal{Y}$ are given in black. Note that we cannot simply apply the double-move from Case 1 in Theorem~\ref{thm:two_exchanges} as depicted for this scenario in Figure~\ref{fig:case_1_modified_a} (compare to Figure~\ref{fig:case_1}). In the first cyclical move, the edge $e_P$ would be applied, sending the item $x_t$ from $w_{t-1}$ to $w_2$. Hence, $w_{t-1}$ would then be unable to send $x_t$ to $w_t$ in the second sequential move. 

We can address this by making a slight modification to this first cyclical move: instead of sending $x_2$ from $w_1$ to $w_{t-1}$ and then applying $e_P$, simply send $x_2$ directly from $w_1$ to $w_2$. Then $x_t$ remains at $w_{t-1}$, and in the second sequential move, $w_{t-1}$ can send $x_t$ to $w_t$ as seen in Figure~\ref{fig:case_1_modified_b}. 

\begin{figure}
\begin{subfigure}{1.0 \textwidth}
\centering
\begin{tikzpicture}[vertices/.style={draw, fill=black, circle, inner sep=0pt, minimum size = 4pt, outer sep=0pt}]

\node[vertices] (c_1) at (0,1.2) {};
\node[vertices] (c_2) at (0.75, 2.4) {};
\node[vertices] (c_3) at (2, 2.4) {};
\node[vertices] (c_4) at (2.75, 1.2) {};
\node[vertices] (c_5) at (2, 0) {};
\node[vertices] (c_6) at (0.75, 0) {};

\node[vertices] (e_1) at (4.25,0.7) {};
\node[vertices] (e_2) at (4.25, 1.8) {};
\node[vertices] (e_3) at (5.35, 1.8) {};
\node[vertices] (e_4) at (5.35, 0.7) {};

\node[vertices] (d_1) at (6.85,1.2) {};
\node[vertices] (d_2) at (7.6, 2.4) {};
\node[vertices] (d_3) at (8.85, 2.4) {};
\node[vertices] (d_4) at (9.6, 1.2) {};
\node[vertices] (d_5) at (8.85, 0) {};
\node[vertices] (d_6) at (7.6, 0) {};

\foreach \to/\from in {
	c_1/c_2, c_2/c_3, c_3/c_4, c_4/c_5, c_5/c_6, c_6/c_1,
	d_1/d_2, d_3/d_4, d_4/d_5, d_5/d_6, d_6/d_1,
	e_1/e_2, e_2/e_3, e_3/e_4, e_4/e_1}
\draw[draw=black, line width= 1,  ->, >=latex]  (\to)--(\from);

\node[above] at (c_2) {$w_1$};
\node[below] at (d_5) {$w_3$};
\node[above] at (d_2) {$w_{t-1}$};
\node[above right] at (d_3) {$w_{2}$};

\path [draw=green,line width= 1,->,>=latex,green] (d_2) edge[bend left=0] (d_3);

\path [draw=blue,line width= 1,->,>=latex,blue] (c_2) edge[bend left=10] (d_2);
\path [draw=blue,line width= 1,->,>=latex,blue] (d_2) edge[bend left=20] (d_3);
\path [draw=blue,line width= 1,->,>=latex,blue] (d_3) edge[bend left=20] (d_4);
\path [draw=blue,line width= 1,->,>=latex,blue] (d_4) edge[bend left=20] (d_5);
\path [draw=blue,line width= 1,->,>=latex,blue] (d_5) edge[bend right=20] (d_6);
\path [draw=blue,line width= 1,->,>=latex,blue] (d_6) edge[bend left=20] (d_1);
\path [draw=blue,line width= 1,->,>=latex,blue] (d_1) edge[bend left=20] (e_3);
\path [draw=blue,line width= 1,->,>=latex,blue] (e_3) edge[bend right=20] (e_4);
\path [draw=blue,line width= 1,->,>=latex,blue] (e_4) edge[bend left=20] (e_1);
\path [draw=blue,line width= 1,->,>=latex,blue] (e_1) edge[bend left=20] (e_2);
\path [draw=blue,line width= 1,->,>=latex,blue] (e_2) edge[bend left=20] (c_3);
\path [draw=blue,line width= 1,->,>=latex,blue] (c_3) edge[bend right=20] (c_4);
\path [draw=blue,line width= 1,->,>=latex,blue] (c_4) edge[bend left=20] (c_5);
\path [draw=blue,line width= 1,->,>=latex,blue] (c_5) edge[bend left=20] (c_6);
\path [draw=blue,line width= 1,->,>=latex,blue] (c_6) edge[bend left=20] (c_1);
\path [draw=blue,line width= 1,->,>=latex,blue] (c_1) edge[bend right=20] (c_2);

\path [draw=red,line width= 1,->,>=latex,red] (c_2) edge[bend right=15] (c_3);
\path [draw=red,line width= 1,->,>=latex,red] (c_3) edge[bend left=15] (e_3);
\path [draw=red,line width= 1,->,>=latex,red] (e_3) edge[bend right=10] (d_2);
\path [draw=red,line width= 1,->,>=latex,red] (d_2) edge[bend left=10] (d_5);

\end{tikzpicture}
\caption{Original double-move from Case 1 of Theorem~\ref{thm:two_exchanges}.}\label{fig:case_1_modified_a}
\end{subfigure}

\begin{subfigure}{1.0 \textwidth}
\centering
\begin{tikzpicture}[vertices/.style={draw, fill=black, circle, inner sep=0pt, minimum size = 4pt, outer sep=0pt}]

\node[vertices] (c_1) at (0,1.2) {};
\node[vertices] (c_2) at (0.75, 2.4) {};
\node[vertices] (c_3) at (2, 2.4) {};
\node[vertices] (c_4) at (2.75, 1.2) {};
\node[vertices] (c_5) at (2, 0) {};
\node[vertices] (c_6) at (0.75, 0) {};

\node[vertices] (e_1) at (4.25,0.7) {};
\node[vertices] (e_2) at (4.25, 1.8) {};
\node[vertices] (e_3) at (5.35, 1.8) {};
\node[vertices] (e_4) at (5.35, 0.7) {};

\node[vertices] (d_1) at (6.85,1.2) {};
\node[vertices] (d_2) at (7.6, 2.4) {};
\node[vertices] (d_3) at (8.85, 2.4) {};
\node[vertices] (d_4) at (9.6, 1.2) {};
\node[vertices] (d_5) at (8.85, 0) {};
\node[vertices] (d_6) at (7.6, 0) {};

\node (x_2) at (4.9, 3.25) {$x_2$};
\node (x_t) at (8.2, 1.2) {$x_t$};

\foreach \to/\from in {
	c_1/c_2, c_2/c_3, c_3/c_4, c_4/c_5, c_5/c_6, c_6/c_1,
	d_1/d_2, d_3/d_4, d_4/d_5, d_5/d_6, d_6/d_1,
	e_1/e_2, e_2/e_3, e_3/e_4, e_4/e_1}
\draw[draw=black, line width= 1,  ->, >=latex]  (\to)--(\from);

\node[above] at (c_2) {$w_1$};
\node[below left] at (c_3) {};
\node[below] at (d_5) {$w_t$};
\node[right] at (d_1) {};
\node[above] at (e_2) {};
\node[above] at (e_3) {};
\node[above] at (d_2) {$w_{t-1}$};
\node[above right] at (d_3) {$w_{2}$};

\path [draw=green,line width= 1,->,>=latex,green] (d_2) edge[bend right=0] (d_3);

\path [draw=blue,line width= 1,->,>=latex,blue] (c_2) edge[bend left=25] (d_3);
\path [draw=blue,line width= 1,->,>=latex,blue] (d_3) edge[bend left=20] (d_4);
\path [draw=blue,line width= 1,->,>=latex,blue] (d_4) edge[bend left=20] (d_5);
\path [draw=blue,line width= 1,->,>=latex,blue] (d_5) edge[bend right=20] (d_6);
\path [draw=blue,line width= 1,->,>=latex,blue] (d_6) edge[bend left=20] (d_1);
\path [draw=blue,line width= 1,->,>=latex,blue] (d_1) edge[bend left=20] (e_3);
\path [draw=blue,line width= 1,->,>=latex,blue] (e_3) edge[bend right=20] (e_4);
\path [draw=blue,line width= 1,->,>=latex,blue] (e_4) edge[bend left=20] (e_1);
\path [draw=blue,line width= 1,->,>=latex,blue] (e_1) edge[bend left=20] (e_2);
\path [draw=blue,line width= 1,->,>=latex,blue] (e_2) edge[bend left=20] (c_3);
\path [draw=blue,line width= 1,->,>=latex,blue] (c_3) edge[bend right=20] (c_4);
\path [draw=blue,line width= 1,->,>=latex,blue] (c_4) edge[bend left=20] (c_5);
\path [draw=blue,line width= 1,->,>=latex,blue] (c_5) edge[bend left=20] (c_6);
\path [draw=blue,line width= 1,->,>=latex,blue] (c_6) edge[bend left=20] (c_1);
\path [draw=blue,line width= 1,->,>=latex,blue] (c_1) edge[bend right=20] (c_2);

\path [draw=red,line width= 1,->,>=latex,red] (c_2) edge[bend right=15] (c_3);
\path [draw=red,line width= 1,->,>=latex,red] (c_3) edge[bend left=15] (e_3);
\path [draw=red,line width= 1,->,>=latex,red] (e_3) edge[bend right=10] (d_2);
\path [draw=red,line width= 1,->,>=latex,red] (d_2) edge[bend left=10] (d_5);

\end{tikzpicture}
\caption{A modified double-move for Case 1.}\label{fig:case_1_modified_b}
\end{subfigure}
\caption{The original and modified double-move for Case 1.}\label{fig:case_1_modified}
\end{figure}
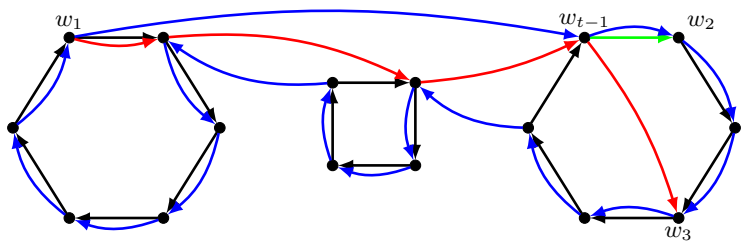
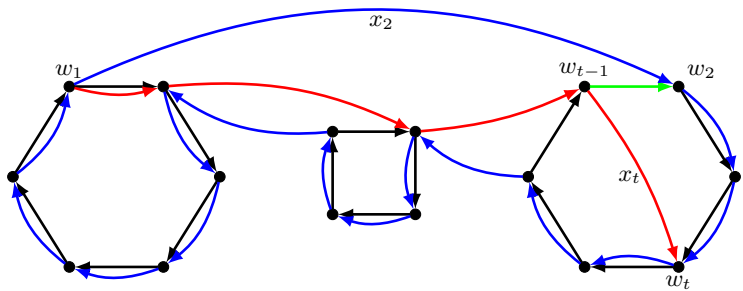

Note that in this modified double-move, the artificial edge $e_P$ from $w_{t-1}$ to $w_2$ is never actually applied. However, its intended purpose is accomplished: item $x_2$ is correctly received by $w_2$ from $w_1$, and $x_t$ is correctly sent from $w_{t-1}$ to $w_t$. Therefore, after the double-move is applied, we can remove $P'$ from $D_P$ and $e_P$ along with the other edges of $\mathcal{Y}$ from $D_Y$, as desired. 

\vspace{.1in}
\textbf{Case 2:} The first and last vertices of $P' = w_1 w_{t-1} w_t$ which are covered by $\mathcal{Y}$ belong to different cycles, and the second-to-last vertex covered by $\mathcal{Y}$ belongs to a different cycle than the last vertex. There are three double-moves based on the double-move from Case 2 of Theorem~\ref{thm:two_exchanges} which can be used depending on whether or not $w_1$, $w_t$, or both $w_1$ and $w_t$ are covered by $\mathcal{Y}$. Depictions of these moves are given in Figure~\ref{fig:case_2_modified}.

\begin{enumerate}[a)]
\item Vertices $w_1$ and $w_t$ are both covered by $\mathcal{Y}$. In this situation, for Case 2 to apply,  $w_1$ and $w_t$ must belong to different cycles of $\mathcal{Y}$ and $w_{t-1}$ must not belong to the same cycle as $w_t$. See Figure~\ref{fig:case_2_modified_a}. Then when performing the double-move from Case 2 of Theorem~\ref{thm:two_exchanges} as depicted in Figure~\ref{fig:case_2}, the edge $e_P$ is applied in the first sequential move before any of the edges from $P'$, sending $x_t$ from $w_{t-1}$ to $w_2$. Hence, if we switch the interior vertex of $P'$ from $w_{t-1}$ to $w_2$, we can apply this original double-move without any further modifications, as depicted in Figure~\ref{fig:case_2_modified_a}. 

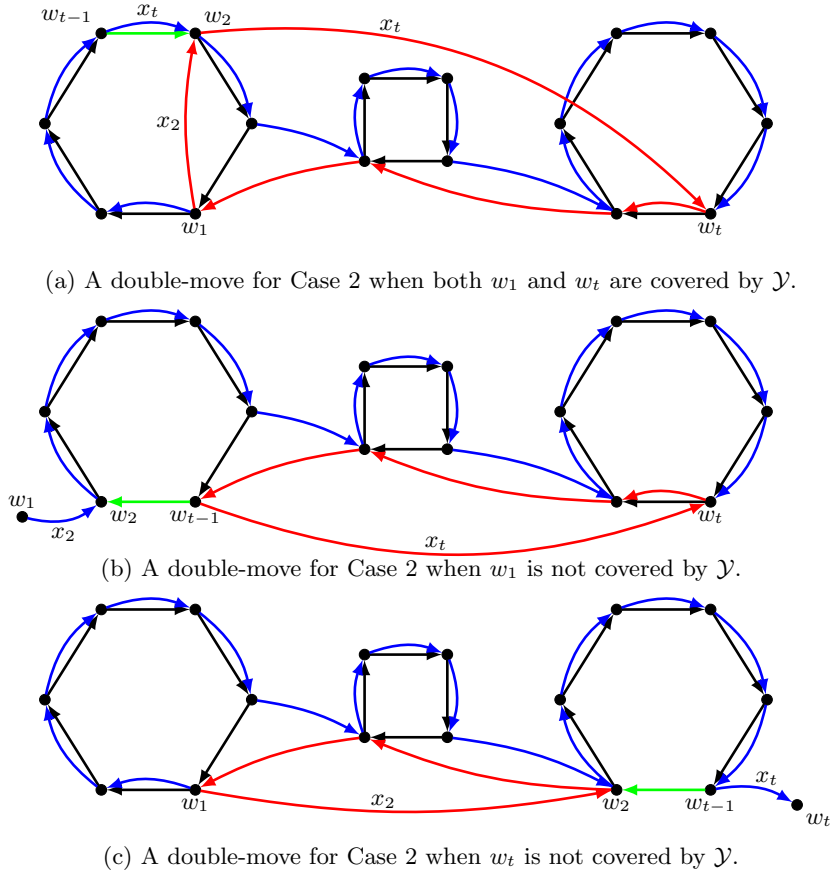
\begin{figure}[t]
\centering
\begin{subfigure}{1.0 \textwidth}
\centering
\begin{tikzpicture}[vertices/.style={draw, fill=black, circle, inner sep=0pt, minimum size = 4pt, outer sep=0pt}]
\useasboundingbox (-0.5,-0.5) rectangle (10.5,2.7);

\node[vertices] (c_1) at (0,1.2) {};
\node[vertices] (c_2) at (0.75, 2.4) {};
\node[vertices] (c_3) at (2, 2.4) {};
\node[vertices] (c_4) at (2.75, 1.2) {};
\node[vertices] (c_5) at (2, 0) {};
\node[vertices] (c_6) at (0.75, 0) {};

\node[vertices] (e_1) at (4.25,0.7) {};
\node[vertices] (e_2) at (4.25, 1.8) {};
\node[vertices] (e_3) at (5.35, 1.8) {};
\node[vertices] (e_4) at (5.35, 0.7) {};

\node[vertices] (d_1) at (6.85,1.2) {};
\node[vertices] (d_2) at (7.6, 2.4) {};
\node[vertices] (d_3) at (8.85, 2.4) {};
\node[vertices] (d_4) at (9.6, 1.2) {};
\node[vertices] (d_5) at (8.85, 0) {};
\node[vertices] (d_6) at (7.6, 0) {};

\node (x_t1) at (1.375, 2.7) {$x_t$};
\node (x_t2) at (4.6, 2.5) {$x_t$};
\node (x_2) at (1.65, 1.2) {$x_2$};

\foreach \to/\from in {
	c_1/c_2, c_3/c_4, c_4/c_5, c_5/c_6, c_6/c_1,
	d_1/d_2, d_2/d_3, d_3/d_4, d_4/d_5, d_5/d_6, d_6/d_1,
	e_1/e_2, e_2/e_3, e_3/e_4, e_4/e_1}
\draw[draw=black, line width= 1,  ->, >=latex]  (\to)--(\from);

\node[below] at (c_5) {$w_{1}$};
\node[below] at (d_5) {$w_{t}$};
\node[above right] at (c_3) {$w_{2}$};
\node[above left] at (c_2) {$w_{t-1}$};

\path [draw=green,line width= 1,->,>=latex,green] (c_2) edge[bend right=0] (c_3);

\path [draw=blue,line width= 1,->,>=latex,blue] (c_5) edge[bend right=20] (c_6);
\path [draw=blue,line width= 1,->,>=latex,blue] (c_6) edge[bend left=20] (c_1);
\path [draw=blue,line width= 1,->,>=latex,blue] (c_1) edge[bend left=20] (c_2);
\path [draw=blue,line width= 1,->,>=latex,blue] (c_2) edge[bend left=20] (c_3);
\path [draw=blue,line width= 1,->,>=latex,blue] (c_3) edge[bend left=20] (c_4);
\path [draw=blue,line width= 1,->,>=latex,blue] (c_4) edge[bend left=10] (e_1);
\path [draw=blue,line width= 1,->,>=latex,blue] (e_1) edge[bend left=20] (e_2);
\path [draw=blue,line width= 1,->,>=latex,blue] (e_2) edge[bend left=20] (e_3);
\path [draw=blue,line width= 1,->,>=latex,blue] (e_3) edge[bend left=20] (e_4);
\path [draw=blue,line width= 1,->,>=latex,blue] (e_4) edge[bend left=10] (d_6);
\path [draw=blue,line width= 1,->,>=latex,blue] (d_6) edge[bend left=20] (d_1);;
\path [draw=blue,line width= 1,->,>=latex,blue] (d_1) edge[bend left=20] (d_2);
\path [draw=blue,line width= 1,->,>=latex,blue] (d_2) edge[bend left=20] (d_3);
\path [draw=blue,line width= 1,->,>=latex,blue] (d_3) edge[bend left=20] (d_4);
\path [draw=blue,line width= 1,->,>=latex,blue] (d_4) edge[bend left=20] (d_5);

\path [draw=red,line width= 1,->,>=latex,red] (d_5) edge[bend right=20] (d_6);
\path [draw=red,line width= 1,->,>=latex,red] (d_6) edge[bend left=10] (e_1);
\path [draw=red,line width= 1,->,>=latex,red] (e_1) edge[bend right=10] (c_5);
\path [draw=red,line width= 1,->,>=latex,red] (c_5) edge[bend left=10] (c_3);
\path [draw=red,line width= 1,->,>=latex,red] (c_3) edge[bend left=25] (d_5);
\end{tikzpicture}
\caption{A double-move for Case 2 when both $w_1$ and $w_t$ are covered by $\mathcal{Y}$.}\label{fig:case_2_modified_a}
\end{subfigure}
\begin{subfigure}{1.0 \textwidth}
\centering
\begin{tikzpicture}[vertices/.style={draw, fill=black, circle, inner sep=0pt, minimum size = 4pt, outer sep=0pt}]

\useasboundingbox (-0.5,-0.5) rectangle (10.5,2.7);

\node[vertices] (c_1) at (0,1.2) {};
\node[vertices] (c_2) at (0.75, 2.4) {};
\node[vertices] (c_3) at (2, 2.4) {};
\node[vertices] (c_4) at (2.75, 1.2) {};
\node[vertices] (c_5) at (2, 0) {};
\node[vertices] (c_6) at (0.75, 0) {};

\node[vertices] (e_1) at (4.25,0.7) {};
\node[vertices] (e_2) at (4.25, 1.8) {};
\node[vertices] (e_3) at (5.35, 1.8) {};
\node[vertices] (e_4) at (5.35, 0.7) {};

\node[vertices] (d_1) at (6.85,1.2) {};
\node[vertices] (d_2) at (7.6, 2.4) {};
\node[vertices] (d_3) at (8.85, 2.4) {};
\node[vertices] (d_4) at (9.6, 1.2) {};
\node[vertices] (d_5) at (8.85, 0) {};
\node[vertices] (d_6) at (7.6, 0) {};

\node[vertices] (w_1) at (-0.3,-0.2) {};

\foreach \to/\from in {
	c_1/c_2, c_2/c_3, c_3/c_4, c_4/c_5, c_6/c_1,
	d_1/d_2, d_2/d_3, d_3/d_4, d_4/d_5, d_5/d_6, d_6/d_1,
	e_1/e_2, e_2/e_3, e_3/e_4, e_4/e_1}
\draw[draw=black, line width= 1,  ->, >=latex]  (\to)--(\from);

\node (x_2) at (0.25, -0.43) {$x_2$};
\node (x_t) at (5.2, -0.55) {$x_t$};

\node[above] at (w_1) {$w_1$};
\node[below] at (c_5) {$w_{t-1}$};
\node[below right] at (c_6) {$w_{2}$};
\node[below] at (d_5) {$w_{t}$};

\path [draw=green,line width= 1,->,>=latex,green] (c_5) edge (c_6);

\path [draw=blue,line width= 1,->,>=latex,blue] (w_1) edge[bend right=25] (c_6);
\path [draw=blue,line width= 1,->,>=latex,blue] (c_6) edge[bend left=20] (c_1);
\path [draw=blue,line width= 1,->,>=latex,blue] (c_1) edge[bend left=20] (c_2);
\path [draw=blue,line width= 1,->,>=latex,blue] (c_2) edge[bend left=20] (c_3);
\path [draw=blue,line width= 1,->,>=latex,blue] (c_3) edge[bend left=20] (c_4);
\path [draw=blue,line width= 1,->,>=latex,blue] (c_4) edge[bend left=10] (e_1);
\path [draw=blue,line width= 1,->,>=latex,blue] (e_1) edge[bend left=20] (e_2);
\path [draw=blue,line width= 1,->,>=latex,blue] (e_2) edge[bend left=20] (e_3);
\path [draw=blue,line width= 1,->,>=latex,blue] (e_3) edge[bend left=20] (e_4);
\path [draw=blue,line width= 1,->,>=latex,blue] (e_4) edge[bend left=10] (d_6);
\path [draw=blue,line width= 1,->,>=latex,blue] (d_6) edge[bend left=20] (d_1);;
\path [draw=blue,line width= 1,->,>=latex,blue] (d_1) edge[bend left=20] (d_2);
\path [draw=blue,line width= 1,->,>=latex,blue] (d_2) edge[bend left=20] (d_3);
\path [draw=blue,line width= 1,->,>=latex,blue] (d_3) edge[bend left=20] (d_4);
\path [draw=blue,line width= 1,->,>=latex,blue] (d_4) edge[bend left=20] (d_5);

\path [draw=red,line width= 1,->,>=latex,red] (d_5) edge[bend right=20] (d_6);
\path [draw=red,line width= 1,->,>=latex,red] (d_6) edge[bend left=10] (e_1);
\path [draw=red,line width= 1,->,>=latex,red] (e_1) edge[bend right=10] (c_5);
\path [draw=red,line width= 1,->,>=latex,red] (c_5) edge[bend right=20] (d_5);

\end{tikzpicture}
\caption{A double-move for Case 2 when $w_1$ is not covered by $\mathcal{Y}.$}\label{fig:case_2_modified_b}
\end{subfigure}
\begin{subfigure}{1.0 \textwidth}
\centering
\begin{tikzpicture}[vertices/.style={draw, fill=black, circle, inner sep=0pt, minimum size = 4pt, outer sep=0pt}]
\useasboundingbox (-0.5,-0.5) rectangle (10.5,2.7);

\node[vertices] (c_1) at (0,1.2) {};
\node[vertices] (c_2) at (0.75, 2.4) {};
\node[vertices] (c_3) at (2, 2.4) {};
\node[vertices] (c_4) at (2.75, 1.2) {};
\node[vertices] (c_5) at (2, 0) {};
\node[vertices] (c_6) at (0.75, 0) {};

\node[vertices] (e_1) at (4.25,0.7) {};
\node[vertices] (e_2) at (4.25, 1.8) {};
\node[vertices] (e_3) at (5.35, 1.8) {};
\node[vertices] (e_4) at (5.35, 0.7) {};

\node[vertices] (d_1) at (6.85,1.2) {};
\node[vertices] (d_2) at (7.6, 2.4) {};
\node[vertices] (d_3) at (8.85, 2.4) {};
\node[vertices] (d_4) at (9.6, 1.2) {};
\node[vertices] (d_5) at (8.85, 0) {};
\node[vertices] (d_6) at (7.6, 0) {};

\node[vertices] (w_2) at (10.0, -0.2) {};

\foreach \to/\from in {
	c_1/c_2, c_2/c_3, c_3/c_4, c_4/c_5, c_5/c_6, c_6/c_1,
	d_1/d_2, d_2/d_3, d_3/d_4, d_4/d_5, d_6/d_1,
	e_1/e_2, e_2/e_3, e_3/e_4, e_4/e_1}
\draw[draw=black, line width= 1,  ->, >=latex]  (\to)--(\from);

\node[below right] at (w_2) {$w_t$};
\node[below] at (c_5) {$w_{1}$};
\node[below] at (d_5) {$w_{t-1}$};
\node[below] at (d_6) {$w_2$};

\node (x_2) at (4.5, -0.15) {$x_2$};
\node (x_t) at (9.6, 0.15) {$x_t$};

\path [draw=green,line width= 1,->,>=latex,green] (d_5) edge (d_6);

\path [draw=blue,line width= 1,->,>=latex,blue] (c_5) edge[bend right=20] (c_6);
\path [draw=blue,line width= 1,->,>=latex,blue] (c_6) edge[bend left=20] (c_1);
\path [draw=blue,line width= 1,->,>=latex,blue] (c_1) edge[bend left=20] (c_2);
\path [draw=blue,line width= 1,->,>=latex,blue] (c_2) edge[bend left=20] (c_3);
\path [draw=blue,line width= 1,->,>=latex,blue] (c_3) edge[bend left=20] (c_4);
\path [draw=blue,line width= 1,->,>=latex,blue] (c_4) edge[bend left=10] (e_1);
\path [draw=blue,line width= 1,->,>=latex,blue] (e_1) edge[bend left=20] (e_2);
\path [draw=blue,line width= 1,->,>=latex,blue] (e_2) edge[bend left=20] (e_3);
\path [draw=blue,line width= 1,->,>=latex,blue] (e_3) edge[bend left=20] (e_4);
\path [draw=blue,line width= 1,->,>=latex,blue] (e_4) edge[bend left=10] (d_6);
\path [draw=blue,line width= 1,->,>=latex,blue] (d_6) edge[bend left=20] (d_1);;
\path [draw=blue,line width= 1,->,>=latex,blue] (d_1) edge[bend left=20] (d_2);
\path [draw=blue,line width= 1,->,>=latex,blue] (d_2) edge[bend left=20] (d_3);
\path [draw=blue,line width= 1,->,>=latex,blue] (d_3) edge[bend left=20] (d_4);
\path [draw=blue,line width= 1,->,>=latex,blue] (d_4) edge[bend left=20] (d_5);
\path [draw=blue,line width= 1,->,>=latex,blue] (d_5) edge[bend left=20] (w_2);

\path [draw=red,line width= 1,->,>=latex,red] (d_6) edge[bend left=10] (e_1);
\path [draw=red,line width= 1,->,>=latex,red] (e_1) edge[bend right=10] (c_5);
\path [draw=red,line width= 1,->,>=latex,red] (c_5) edge[bend right=10] (d_6);

\end{tikzpicture}
\caption{A double-move for Case 2 when $w_t$ is not covered by $\mathcal{Y}$.}\label{fig:case_2_modified_c}
\end{subfigure}
\caption{Modified double-moves for Case 2.}\label{fig:case_2_modified}
\end{figure}

\item Vertex $w_1$ is not covered by $\mathcal{Y}$. Then for Case 2 to apply, $w_{t-1}$ and $w_t$ must belong to different cycles of $\mathcal{Y}$. See Figure~\ref{fig:case_2_modified_b}. As in Case 1, we cannot apply the original double-move since then $w_{t-1}$ would be unable to send $x_t$ to $w_t$ in the second move. We can address this in the same way as in the modified double-move from Case 1: send $x_2$ directly from $w_1$ to $w_2$ in the first sequential move and then send $x_t$ directly from $w_{t-1}$ to $w_t$ in the second cyclical move, as depicted in Figure~\ref{fig:case_2_modified_b}. Although $e_P$ is never actually applied, all desired transfers are accomplished as in Case 1.

\item Vertex $w_t$ is not covered by $\mathcal{Y}$. Then for Case 2 to apply, $w_1$ and $w_{t-1}$ must belong to different cycles of $\mathcal{Y}$. See Figure~\ref{fig:case_2_modified_c}. We make a similar modification to that of the previous case: correctly send $x_t$ directly from $w_{t-1}$ to $w_t$ in the first sequential move and then correctly send $x_2$ directly from $w_1$ to $w_2$ in the second cyclical move, as depicted in Figure~\ref{fig:case_2_modified_c}.
\end{enumerate}

\textbf{Case 3:} The first and last vertices of $P = w_1 w_{t-1} w_t$ which are covered by $\mathcal{Y}$ belong to the same cycle in $\mathcal{Y}$, while the second-to-last vertex belongs to a different cycle. The only scenario in which this case applies is when $w_1$ and $w_t$ belong to the same cycle of $\mathcal{Y}$ and $w_{t-1}$ belongs to a different cycle. We make a modification similar to the third double-move from the previous case. In a first cyclical move send $x_t$ directly from $w_{t-1}$ to $w_t$, and in a second sequential move send $x_2$ directly from $w_1$ to $w_2$. See Figure~\ref{fig:case_3_modified}.

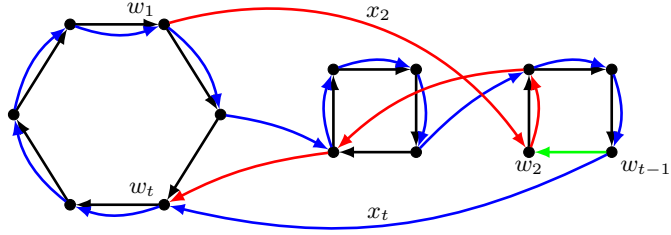
\begin{figure}[t]
\centering
\begin{tikzpicture}[vertices/.style={draw, fill=black, circle, inner sep=0pt, minimum size = 4pt, outer sep=0pt}]
\useasboundingbox (-0.25,-0.25) rectangle (8.5,2.7);

\node[vertices] (c_1) at (0,1.2) {};
\node[vertices] (c_2) at (0.75, 2.4) {};
\node[vertices] (c_3) at (2, 2.4) {};
\node[vertices] (c_4) at (2.75, 1.2) {};
\node[vertices] (c_5) at (2, 0) {};
\node[vertices] (c_6) at (0.75, 0) {};

\node[vertices] (e_1) at (4.25,0.7) {};
\node[vertices] (e_2) at (4.25, 1.8) {};
\node[vertices] (e_3) at (5.35, 1.8) {};
\node[vertices] (e_4) at (5.35, 0.7) {};

\node[vertices] (d_1) at (6.85,0.7) {};
\node[vertices] (d_2) at (6.85, 1.8) {};
\node[vertices] (d_3) at (7.95, 1.8) {};
\node[vertices] (d_4) at (7.95, 0.7) {};

\foreach \to/\from in {
	c_1/c_2, c_2/c_3, c_3/c_4, c_4/c_5, c_5/c_6, c_6/c_1,
	d_1/d_2, d_2/d_3, d_3/d_4,
	e_1/e_2, e_2/e_3, e_3/e_4, e_4/e_1}
\draw[draw=black, line width= 1,  ->, >=latex]  (\to)--(\from);

\node (x_2) at (4.85, 2.55) {$x_2$};
\node (x_2) at (4.85, -0.12) {$x_t$};

\node[above left] at (c_3) {$w_{1}$};
\node[above left] at (c_5) {$w_{t}$};;
\node[below] at (d_1) {$w_{2}$};
\node[below right] at (d_4) {$w_{t-1}$};
\path [draw=green,line width= 1,->,>=latex,green] (d_4) edge[bend right=0] (d_1);

\path [draw=blue,line width= 1,->,>=latex,blue] (c_5) edge[bend left=20] (c_6);
\path [draw=blue,line width= 1,->,>=latex,blue] (c_6) edge[bend left=20] (c_1);
\path [draw=blue,line width= 1,->,>=latex,blue] (c_1) edge[bend left=20] (c_2);
\path [draw=blue,line width= 1,->,>=latex,blue] (c_2) edge[bend right=20] (c_3);
\path [draw=blue,line width= 1,->,>=latex,blue] (c_3) edge[bend left=20] (c_4);
\path [draw=blue,line width= 1,->,>=latex,blue] (c_4) edge[bend left=10] (e_1);
\path [draw=blue,line width= 1,->,>=latex,blue] (e_1) edge[bend left=20] (e_2);
\path [draw=blue,line width= 1,->,>=latex,blue] (e_2) edge[bend left=20] (e_3);
\path [draw=blue,line width= 1,->,>=latex,blue] (e_3) edge[bend left=20] (e_4);
\path [draw=blue,line width= 1,->,>=latex,blue] (e_4) edge[bend left=10] (d_2);
\path [draw=blue,line width= 1,->,>=latex,blue] (d_2) edge[bend left=20] (d_3);
\path [draw=blue,line width= 1,->,>=latex,blue] (d_3) edge[bend left=20] (d_4);
\path [draw=blue,line width= 1,->,>=latex,blue] (d_4) edge[bend left=20] (c_5);

\path [draw=red,line width= 1,->,>=latex,red] (c_3) edge[bend left=35] (d_1);
\path [draw=red,line width= 1,->,>=latex,red] (d_1) edge[bend right=20] (d_2);
\path [draw=red,line width= 1,->,>=latex,red] (d_2) edge[bend right=20] (e_1);
\path [draw=red,line width= 1,->,>=latex,red] (e_1) edge[bend right=10] (c_5);

\end{tikzpicture}
\caption{A modified double-move for Case 3.}\label{fig:case_3_modified}
\end{figure}

\textbf{Case 4:} All vertices of $P' = w_1 w_{t-1} w_t$ which are covered by $\mathcal{Y}$ belong to the same cycle. There are two double-moves, depicted in Figure~\ref{fig:case_4_modified}, which can be used depending on whether or not $w_1$ is covered by $\mathcal{Y}$.

\begin{enumerate}[a)]

\item Vertex $w_1$ is covered by $\mathcal{Y}$. As in the first double-move for Case 2, in the original double-move for Case 4 of Theorem~\ref{thm:two_exchanges} the edge $e_P$ is applied before any of the edges from $P'$. Hence, if we switch the interior vertex of $P'$ from $w_{t-1}$ to $w_2$, we can apply the double-move without any further modifications, as depicted in Figure~\ref{fig:case_4_modified_a}. Note that $w_t$ may or may not be covered by the cycle containing $w_1$ and $e_P$. 

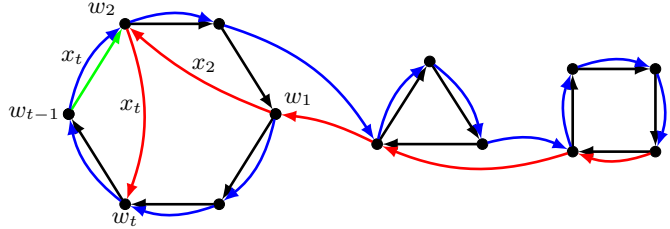
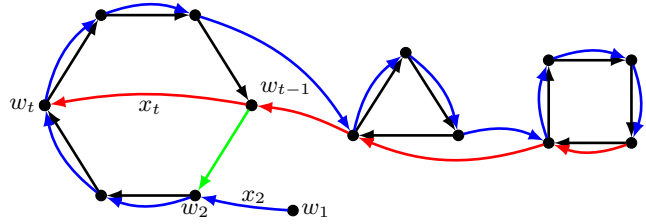
\begin{figure}[t]
\centering
\begin{subfigure}{1.0\textwidth}
\centering
\begin{tikzpicture}[vertices/.style={draw, fill=black, circle, inner sep=0pt, minimum size = 4pt, outer sep=0pt}]
\useasboundingbox (-0.25,-0.25) rectangle (8.5,2.7);

\node[vertices] (c_1) at (0,1.2) {};
\node[vertices] (c_2) at (0.75, 2.4) {};
\node[vertices] (c_3) at (2, 2.4) {};
\node[vertices] (c_4) at (2.75, 1.2) {};
\node[vertices] (c_5) at (2, 0) {};
\node[vertices] (c_6) at (0.75, 0) {};

\node[vertices] (w_2) at (2.75, 1.2) {};
\node[vertices] (w_4) at (0, 1.2) {};

\node[vertices] (d_1) at (4.1,0.8) {};
\node[vertices] (d_2) at (4.8, 1.9) {};
\node[vertices] (d_3) at (5.5, 0.8) {};

\node[vertices] (e_1) at (6.7,0.7) {};
\node[vertices] (e_2) at (6.7, 1.8) {};
\node[vertices] (e_3) at (7.8, 1.8) {};
\node[vertices] (e_4) at (7.8, 0.7) {};

\foreach \to/\from in {
	c_2/c_3, c_3/c_4, c_4/c_5, c_5/c_6, c_6/c_1,
	d_1/d_2, d_2/d_3, d_3/d_1,
	e_1/e_2, e_2/e_3, e_3/e_4, e_4/e_1}
\draw[draw=black, line width= 1,  ->, >=latex]  (\to)--(\from);

\node[above right] at (w_2) {$w_{1}$};
\node[below] at (c_6) {$w_{t}$};
\node[above left] at (c_2) {$w_{2}$};
\node[left] at (c_1) {$w_{t-1}$};

\node (x_2) at (0.05, 1.95) {$x_t$};
\node (x_2) at (1.8, 1.85) {$x_2$};
\node (x_2) at (0.83, 1.3) {$x_t$};

\path [draw=green,line width= 1,->,>=latex,green] (c_1) edge[bend right=0] (c_2);

\path [draw=blue,line width= 1,->,>=latex,blue] (c_4) edge[bend left=20] (c_5);
\path [draw=blue,line width= 1,->,>=latex,blue] (c_5) edge[bend left=20] (c_6);
\path [draw=blue,line width= 1,->,>=latex,blue] (c_6) edge[bend left=20] (c_1);
\path [draw=blue,line width= 1,->,>=latex,blue] (c_1) edge[bend left=20] (c_2);
\path [draw=blue,line width= 1,->,>=latex,blue] (c_2) edge[bend left=20] (c_3);
\path [draw=blue,line width= 1,->,>=latex,blue] (c_3) edge[bend left=20] (d_1);
\path [draw=blue,line width= 1,->,>=latex,blue] (d_1) edge[bend left=20] (d_2);
\path [draw=blue,line width= 1,->,>=latex,blue] (d_2) edge[bend left=20] (d_3);
\path [draw=blue,line width= 1,->,>=latex,blue] (d_3) edge[bend left=20] (e_1);
\path [draw=blue,line width= 1,->,>=latex,blue] (e_1) edge[bend left=20] (e_2);
\path [draw=blue,line width= 1,->,>=latex,blue] (e_2) edge[bend left=20] (e_3);
\path [draw=blue,line width= 1,->,>=latex,blue] (e_3) edge[bend left=20] (e_4);

\path [draw=red,line width= 1,->,>=latex,red] (e_4) edge[bend left=20] (e_1);
\path [draw=red,line width= 1,->,>=latex,red] (e_1) edge[bend left=20] (d_1);
\path [draw=red,line width= 1,->,>=latex,red] (d_1) edge[bend right=10] (w_2);
\path [draw=red,line width= 1,->,>=latex,red] (w_2) edge[bend left=10] (c_2);
\path [draw=red,line width= 1,->,>=latex,red] (c_2) edge[bend left=20] (c_6);

\end{tikzpicture}
\caption{A double-move for Case 4 when $w_1$ is covered by $\mathcal{Y}$.}\label{fig:case_4_modified_a}
\end{subfigure}

\begin{subfigure}{1.0 \textwidth}
\centering
\begin{tikzpicture}[vertices/.style={draw, fill=black, circle, inner sep=0pt, minimum size = 4pt, outer sep=0pt}]
\useasboundingbox (-0.25,-0.25) rectangle (8.5,2.7);

\node[vertices] (c_1) at (0,1.2) {};
\node[vertices] (c_2) at (0.75, 2.4) {};
\node[vertices] (c_3) at (2, 2.4) {};
\node[vertices] (c_4) at (2.75, 1.2) {};
\node[vertices] (c_5) at (2, 0) {};
\node[vertices] (c_6) at (0.75, 0) {};

\node[vertices] (w_1) at (3.3,-0.2) {};
\node[vertices] (w_2) at (2.75, 1.2) {};
\node[vertices] (w_4) at (0, 1.2) {};

\node[vertices] (d_1) at (4.1,0.8) {};
\node[vertices] (d_2) at (4.8, 1.9) {};
\node[vertices] (d_3) at (5.5, 0.8) {};

\node[vertices] (e_1) at (6.7,0.7) {};
\node[vertices] (e_2) at (6.7, 1.8) {};
\node[vertices] (e_3) at (7.8, 1.8) {};
\node[vertices] (e_4) at (7.8, 0.7) {};

\foreach \to/\from in {
	c_1/c_2, c_2/c_3, c_3/c_4, c_5/c_6, c_6/c_1,
	d_1/d_2, d_2/d_3, d_3/d_1,
	e_1/e_2, e_2/e_3, e_3/e_4, e_4/e_1}
\draw[draw=black, line width= 1,  ->, >=latex]  (\to)--(\from);

\node (x_2) at (2.78, -0.02) {$x_2$};
\node (x_t) at (1.4, 1.17) {$x_t$};

\node[right] at (w_1) {$w_1$};
\node[above right] at (w_2) {$w_{t-1}$};
\node[below] at (c_5) {$w_{2}$};
\node[left] at (c_1) {$w_{t}$};

\path [draw=green,line width= 1,->,>=latex,green] (c_4) edge[bend left=0] (c_5);

\path [draw=blue,line width= 1,->,>=latex,blue] (w_1) edge[bend left=5] (c_5);
\path [draw=blue,line width= 1,->,>=latex,blue] (c_5) edge[bend left=20] (c_6);
\path [draw=blue,line width= 1,->,>=latex,blue] (c_6) edge[bend left=20] (c_1);
\path [draw=blue,line width= 1,->,>=latex,blue] (c_1) edge[bend left=20] (c_2);
\path [draw=blue,line width= 1,->,>=latex,blue] (c_2) edge[bend left=20] (c_3);
\path [draw=blue,line width= 1,->,>=latex,blue] (c_3) edge[bend left=20] (d_1);
\path [draw=blue,line width= 1,->,>=latex,blue] (d_1) edge[bend left=20] (d_2);
\path [draw=blue,line width= 1,->,>=latex,blue] (d_2) edge[bend left=20] (d_3);
\path [draw=blue,line width= 1,->,>=latex,blue] (d_3) edge[bend left=20] (e_1);
\path [draw=blue,line width= 1,->,>=latex,blue] (e_1) edge[bend left=20] (e_2);
\path [draw=blue,line width= 1,->,>=latex,blue] (e_2) edge[bend left=20] (e_3);
\path [draw=blue,line width= 1,->,>=latex,blue] (e_3) edge[bend left=20] (e_4);

\path [draw=red,line width= 1,->,>=latex,red] (e_4) edge[bend left=20] (e_1);
\path [draw=red,line width= 1,->,>=latex,red] (e_1) edge[bend left=20] (d_1);
\path [draw=red,line width= 1,->,>=latex,red] (d_1) edge[bend right=10] (w_2);
\path [draw=red,line width= 1,->,>=latex,red] (w_2) edge[bend right=10] (c_1);

\end{tikzpicture}
\caption{A double-move for Case 4 when $w_1$ is not covered by $\mathcal{Y}$.}\label{fig:case_4_modified_b}
\end{subfigure}
\caption{Modified double-moves for Case 4.}\label{fig:case_4_modified}
\end{figure}

\item Vertex $w_1$ is not covered by $\mathcal{Y}$. We make a modification similar to that of the second double-move for Case 2. In a first sequential move, directly send $x_2$ from $w_1$ to $w_2$, and in a second sequential move, directly send $x_t$ from $w_{t-1}$ to $w_{t}$. See Figure~\ref{fig:case_4_modified_b}. Note again that $w_{t}$ may or may not be covered by the cycle containing $e_P$. 
\end{enumerate}

In each case, we are able to integrate $P'$ with $\mathcal{Y}$ and apply all necessary transfers in only two moves. Therefore, at most $\eta_{i_2}$ double-moves are needed to reduce the shared degree of $D_Y$ to zero, and through each of these double-moves, we remove one of the $\frac{1}{2}\sum_{i=1}^k \delta_i$ paths from $D_P$. Afterwards, we may simply apply the remaining paths in $D_P$, if any, individually. The total number of moves used to transform $\C$ into $\C'$ is thus at most
\begin{align*}
(\eta_{i_1} - \eta_{i_2}) + 2 \eta_{i_2} + \max \left\lbrace \frac{1}{2}\sum_{i=1}^k \delta_i - \eta_{i_2} , \ 0 \right\rbrace = \eta_{i_1} + \max \left\lbrace \eta_{i_2}, \ \frac{1}{2}\sum_{i=1}^k \delta_i \right\rbrace.
\end{align*} \eoproof
\end{proof}

\section{Circuit Diameter of Partition Polytopes}\label{sec:diameter}

A fundamental open question in linear programming is whether or not there exists a polynomial pivot rule for the simplex method. The existence of such a pivot rule would require that the \textit{polynomial Hirsch conjecture} \cite{ks-10} holds; i.e., that the combinatorial diameter of a polyhedron can be polynomially bounded. A recent effort to better understand the combinatorial diameter of polyhedra has been the study of the related circuit diameter \cite{bfh-14,bdf-16,bsy-18,kps-17}. Whereas the original Hirsch conjecture is false in general \cite{kw-67,s-11}, the related \textit{Circuit Diameter Conjecture} \cite{bfh-14} remains open.

Recall that the circuits of the bounded-size partition polytope $BPP$ correspond to cyclical and sequential moves of items among clusters. Therefore, as long as no cluster size constraints are violated during a clustering transformation, any resulting bounds on the transformation distance between clusterings have implications on the circuit distance between vertices in $BPP$. As a 0/1-polytope, the combinatorial diameter (and hence, also the circuit diameter) of $BPP$ satisfies the Hirsch conjecture \cite{n-89} -- specifically, the combinatorial diameter is at most the number of items $n$. In this section, we will use the results from Section~\ref{sec:bounds} to achieve much better upper bounds on the circuit diameter.

For the fixed-size partition polytope, Proposition~\ref{prop:disjoint_cycles} can be used to show that the combinatorial diameter is at most $\kappa_1 + \kappa_2$, where $\kappa_1, \kappa_2$ are the two largest fixed cluster sizes \cite{b-13}. We begin by generalizing this bound to the circuit diameter of the bounded-size partition polytope by also taking into account the largest possible change in cluster sizes. Although we do not yet utilize any double-moves which integrate sequential and cyclical moves (see the upcoming Theorem~\ref{thm:improved_diameter}), the bound of the following lemma is already better than the naive bound achieved by simply counting the sequential and cyclical moves separately -- we can relate the shared degree of a vertex in a $CDG$ to the change in size of the corresponding cluster.

\begin{lemma}\label{lem:naive_diameter}
For a bounded-size partition polytope $BPP(\kappa^+, \kappa^-)$, assume the corresponding clusters are indexed so that $\kappa_1^+ \geq \cdots \geq \kappa_k^+$ and let $i_1, i_2$ denote the two indices minimizing $\kappa_i^+ - \kappa_i^-$. Then the circuit diameter of $BPP(\kappa^+, \kappa^-)$ is at most
\begin{align*}
\kappa^+_{1} + \kappa^+_{2} + \frac{1}{2}\sum_{i \neq i_1, i_2}^k (\kappa_i^+ - \kappa_i^-).
\end{align*}
\end{lemma}

\begin{proof}
Let $\mathcal{C}, \mathcal{C}'$ be $k$-clusterings corresponding to vertices $y, y'$ of $\PP$. We can transform $\mathcal{C}$ into $\mathcal{C}'$ by separately applying sequential moves followed by cyclical double-moves in the manner of Lemma~\ref{lem:naive_distance_bound}. All intermediate clusterings in this transformation satisfy the cluster size bounds of $\PP$, so the process indeed corresponds to a circuit walk from $y$ to $y'$ in $\PP$. 

Let $\eta_i$ denote the shared degree of vertex $c_i$ in $CDG(C ,C')$, and let $\delta_i := |(|C_i| - |C'_i|)|$. Lemma~\ref{lem:naive_distance_bound} then implies that the circuit distance from $y$ to $y'$ in $\PP$ is at most 
\begin{align}
\eta_{j_1} + \eta_{j_2} + \frac{1}{2}\sum_{i=1}^k \delta_i,
\end{align} 
where $j_1, j_2$ maximize $\eta_i$ over all $i = 1,...,k$. Trivially, for $i=1,...k$, it holds that $\eta_i \leq \kappa_i^+$ and $\delta_i \leq \kappa_i^+ - \kappa_i^-$. Hence, we obtain the following upper bound on the circuit diameter of $\PP$ as a natural implication of Lemma~\ref{lem:naive_distance_bound}:
\begin{align*}
\kappa^+_{j_1} + \kappa^+_{j_2} + \frac{1}{2}\sum_{i=1}^k (\kappa_i^+ - \kappa_i^-).
\end{align*}

Note however that this bound can be immediately improved. For $i=1,...,k$, we must have $\eta_i + \delta_i \leq \kappa^+_i$ since $\eta_i + \delta_i$ is equal to the maximum of the indegree and outdegree of $c_i$. Rearranging this inequality yields $\eta_i + \frac{1}{2}\delta_i \leq \kappa_i^+ - \frac{1}{2}\delta_i$. Substituting into (1), we obtain the following upper bound on the circuit distance from $y$ to $y'$: 
\begin{align*}
\eta_{j_1} + \eta_{j_2} + \frac{1}{2}\sum_{i=1}^k \delta_i &= \sum_{i = j_1, j_2} \left(\eta_{i} + \frac{1}{2}\delta_i\right) + \frac{1}{2} \sum_{i \neq j_1, j_2}^k \delta_i \\   
&\leq \sum_{i = j_1, j_2} \left(\kappa^+_i - \frac{1}{2}\delta_i \right) + \frac{1}{2} \sum_{i \neq j_1, j_2}^k \delta_i.
\end{align*}
Note that $\sum_{i = j_1, j_2} \left(\kappa^+_i - \frac{1}{2}\delta_i \right) \leq \kappa^+_{j_1} + \kappa^+_{j_2} \leq \kappa^+_{1} + \kappa^+_{2}$. Similarly, it holds that $\frac{1}{2} \sum_{i \neq j_1, j_2}^k \delta_i \leq \frac{1}{2} \sum_{i \neq j_1, j_2}^k (\kappa_i^+ - \kappa_i^-) \leq \frac{1}{2}\sum_{i \neq i_1, i_2}^k (\kappa_i^+ - \kappa_i^-)$. Thus, we obtain the stated bound. \eoproof
\end{proof}

As in Theorem~\ref{thm:improved_bound}, we can significantly improve upon this diameter bound by using the double-moves from Theorem~\ref{thm:two_exchanges} to integrate sequential moves with sets of disjoint cyclical moves. Note that we must take care when applying these double-moves to bounded-size clusterings -- certain moves require the existence of a vertex whose cluster size can be temporarily increased as demonstrated in Figure~\ref{fig:size_increase}. Nevertheless, as long as there is at least some slack in the constraints for all but at most one cluster, we can ensure the existence of such a vertex through a simple pre-processing of the clusters. Hence, we obtain the following improved diameter bound as an implication of the transformation distance bound from Theorem~\ref{thm:improved_bound}, which depends on the maximum of the second-largest cluster size and the largest possible change in cluster sizes. We note that the cluster size slackness assumptions made in this theorem are quite natural in any application involving bounded-size clusterings.

\begin{theorem}\label{thm:improved_diameter}
For a bounded-size partition polytope $BPP(\kappa^+, \kappa^-)$, assume the corresponding clusters are indexed so that $\kappa_1^+ \geq \cdots \geq \kappa_k^+$ and let $i_1$ denote the index minimizing $\kappa_i^+ - \kappa_i^-$. If $\sum_{i=1}^k \kappa_i^+ > n + k - 2$ and if $\kappa_i^+  > \kappa_i^-$ for $i \neq i_1$, the circuit diameter of $\PP$ is at most
\begin{align*}
\kappa^+_{1} +  \max \left\lbrace \kappa^+_{2}, \ \frac{1}{2}\sum_{i \neq i_1}^k (\kappa_i^+ - \kappa_i^-) \right\rbrace + 2(k-2).
\end{align*}
\end{theorem}

\begin{proof}
Let $\mathcal{C}, \mathcal{C}'$ be $k$-clusterings corresponding to vertices $y, y'$ of $\PP$. We can transform $\mathcal{C}$ into $\mathcal{C}'$ in the manner of Theorem~\ref{thm:improved_bound}. However, in order for all intermediate clusterings to satisfy the bounds of $\PP$, we must make sure that when applying any version of the double-move from Case 4 of Theorem~\ref{thm:two_exchanges}, there exists a suitable choice for $u_s$ whose corresponding cluster size is strictly less than its upper bound and can be temporarily increased. 

To ensure that this is always the case, we pre-process $\mathcal{C}$ and $\mathcal{C}'$ in the following manner. If there exists more than one cluster $C_i$ (or $C'_i$ in the case of $\mathcal{C}'$) such that $|C_i| = \kappa^+_i$, then choose such an index $j$ with $|C_j| = \kappa^+_j > \kappa^-_j$, which is possible since at most one index $i$ satisfies $\kappa^+_i = \kappa^-_i$. Transfer any item from $C_j$ to a different cluster $C_\ell$ which satisfies $|C_\ell| < \kappa^+_\ell - 1$. Such an index $\ell$ must exist, else we would have
\begin{align*}
\sum_{i=1}^k |C_i| \geq 2 + \sum_{i=1}^k (\kappa^+_i - 1) > 2 + (n + k - 2) - k = n.
\end{align*}
Repeat this process at most $k-2$ times until sizes of all clusters but at most one are strictly less than their upper bounds. In the resulting $k$-clustering, it is then always possible to choose $u_s$ in Case 4 of Theorem~\ref{thm:two_exchanges} such that the corresponding cluster size can be temporarily increased when performing the double-move: any cycle in a clustering-difference graph covers at least two vertices, and at least one of these vertices must have a corresponding cluster size less than its upper bound. Additionally, the choice of the vertex $u_s$ is not affected by the modifications in Case 4 of Theorem~\ref{thm:improved_bound}. 

Pre-processing both $\mathcal{C}$ and $\C'$ in this manner requires at most $2(k-2)$ single item transfers. Once these transfers have been applied, we may perform a circuit walk between the resulting $k$-clusterings by applying the moves of Theorem~\ref{thm:improved_bound} to their clustering-difference graph. Such a circuit walk has length at most
\begin{align*}
\eta_{j_1} +  \max \left\lbrace  \eta_{j_2}, \  \frac{1}{2}\sum_{i=1}^k  \delta_i   \right\rbrace ,
\end{align*}
where $j_1, j_2$ denote the two indices maximizing the shared degree $\eta_i$.

As in the proof of Lemma~\ref{lem:naive_diameter}, since $\eta_{j_1} + \frac{1}{2} \delta_{j_1} \leq \kappa_{j_1}^+ - \frac{1}{2}\delta_{j_1}$, this bound is at most
\begin{align*}
\max \left\lbrace  \eta_{j_1} + \eta_{j_2}, \ \left(\eta_{j_1} + \frac{1}{2} \delta_{j_1} \right) + \frac{1}{2}\sum_{i \neq j_1}^k  \delta_i   \right\rbrace &\leq 
\max \left\lbrace \eta_{j_1} + \eta_{j_2}, \  \left(\kappa^+_{j_1} - \frac{1}{2}\delta_{j_1}\right) + \frac{1}{2}\sum_{i \neq j_1}^k \delta_i \right \rbrace  \\
&\leq \max \left\lbrace \kappa^+_{j_1} + \kappa^+_{j_2}, \  \kappa^+_{j_1} + \frac{1}{2}\sum_{i \neq j_1}^k (\kappa_i^+ - \kappa_i^-) \right \rbrace \\
&\leq \kappa^+_1   +  \max \left\lbrace  \kappa^+_2, \ \frac{1}{2}\sum_{i \neq i_1}^k (\kappa_i^+ - \kappa_i^-) \right\rbrace.
\end{align*}
Taking into account the at most $2(k-2)$ circuit steps needed to adjust the cluster sizes, we obtain the stated improved bound. \eoproof
\end{proof}

\section{Conclusions and Future Directions}\label{sec:future_directions}

In this work, we provide methods based on linear programming and network theory for transforming $k$-clusterings using sequences of cyclical and sequential moves of items among clusters. This leads to upper bounds on the transformation distance between two general $k$-clusterings as well as the circuit diameter of the bounded-size partition polytope. There are several natural directions for future research in this area. 

We prove in Theorem~\ref{thm:improved_diameter} an upper bound on the circuit diameter of the bounded-size partition polytope using the transformation distance bound from Theorem~\ref{thm:improved_bound} and modified double-moves from Theorem~\ref{thm:two_exchanges} which integrate sequential moves of items with cyclical moves. A subsequent research question is whether or not we can also bound the combinatorial diameter of the polytope in such a manner. The edges of $BPP$ have a more technical characterization than its circuits -- only certain cyclical and sequential moves actually correspond to edges between vertices \cite{bv-17}. However, through a careful ordering of cyclical and sequential moves and double-moves, we believe new bounds on the combinatorial distance between vertices in the polytope could be achievable. 

Additionally, in Theorem~\ref{thm:improved_bound}, we use an \textit{arbitrary} path-cycle decomposition $(D_P, D_Y)$ of the clustering-difference graph $D := \CDG$ to bound the transformation distance between the clusterings. It is possible to instead construct a decomposition exhibiting potentially useful properties. For instance, solving a minimum-cost circulation problem over $D$ yields a decomposition in which $D_Y$ has a maximum number of edges. Modifying this circulation problem can yield a decomposition in which the maximum shared degree in $D_P$ is minimized. Through further analysis, these extremal choices for the path-cycle decomposition might lead to better upper bounds on the transformation distance.

Finally, we note that the transformation distance $d(\C, \C')$ is formally a metric. Hence, if we are able to compute $d(\C, \C')$, we can interpret it as a measure of the distance between given $k$-clusterings of the same data set. There is significant interest in comparing clusterings in the literature \cite{g-17,m-07}. However, most measures typically do not take into account the potential labels of the clusters and are instead based on pairwise relationships among the items. Our new metric takes a fundamentally different approach to measuring the difference between clusterings, motivating a comparative study.

\bibliography{literature}
\bibliographystyle{plain}

\end{document}